\documentclass[11pt]{amsart}
\usepackage[utf8]{inputenc}
\usepackage[T1]{fontenc}
\usepackage{amsmath,amssymb,amsthm}
\usepackage[margin=1.15in]{geometry}
\usepackage[colorlinks=true,linkcolor=blue,citecolor=blue,urlcolor=blue]{hyperref}

\newtheorem{theorem}{Theorem}[section]
\newtheorem{proposition}[theorem]{Proposition}
\newtheorem{lemma}[theorem]{Lemma}
\newtheorem{corollary}[theorem]{Corollary}
\newtheorem*{theoremA}{Theorem A}
\newtheorem*{theoremB}{Theorem B}
\newtheorem*{theoremC}{Theorem C}
\theoremstyle{definition}
\newtheorem{definition}[theorem]{Definition}
\newtheorem{example}[theorem]{Example}
\newtheorem{question}[theorem]{Question}
\theoremstyle{remark}
\newtheorem{remark}[theorem]{Remark}

\newcommand{\C}{\mathbb C}
\newcommand{\R}{\mathbb R}
\newcommand{\Z}{\mathbb Z}
\newcommand{\PP}{\mathbb P}
\newcommand{\cS}{\mathcal S}
\newcommand{\cF}{\mathcal F}
\DeclareMathOperator{\Ad}{Ad}
\DeclareMathOperator{\diag}{diag}
\DeclareMathOperator{\Gr}{Gr}
\DeclareMathOperator{\Poly}{Poly}

\title{Non-abelian LVM manifolds}
\author{Jos\'e Seade}
\address{Instituto de Matem\'aticas-Cuernavaca, Universidad Nacional Aut\'onoma de
M\'exico.}
\email{jseade@im.unam.mx}
\date{}

\begin{document}

\begin{abstract}
We extend the tangency construction of L\'opez de Medrano, Verjovsky
and Meersseman from diagonal abelian actions to non-abelian Samelson
groups. For these groups the tangency locus is a moment-map zero set:
the commutator term in the leafwise Hessian vanishes there, and
under an admissibility hypothesis, Kempf--Ness convexity gives a global
transversal whose projectivization is a compact complex manifold. The construction includes compact
rotation groups and groups with non-closed central rotations.
In the latter setting we construct explicit families of complex
dimension $2m-3$, $m\geq5$, which are $4$-connected and admit no
symplectic form. Their meromorphic function fields are rational of
transcendence degree $m-2$, although every holomorphic map to a
projective space is constant. We compute their connected automorphism
groups and show that they admit no maximal torus action. They are
Hopf-surface bundles over compact complex quotients of stable
Grassmannian loci; these bases give related examples already in
complex dimension five. The families demonstrate that the non-abelian
tangency construction reaches beyond compact complex manifolds with
maximal torus actions.
\end{abstract}

\maketitle

\section{Introduction}

The LVM manifolds were introduced by L\'opez de Medrano and Verjovsky
\cite{LdMV} and, in full generality, by Meersseman \cite{M}. One
starts from a diagonal linear action of $\C^k$ on $\C^n$ and looks at
the points where its orbits are tangent to the round spheres centered at the origin. Under
the \emph{admissibility} condition of Meersseman \cite{M}, recalled in
Definition~\ref{def:Meersseman-admissible}, this tangency locus is a
smooth global transversal to the foliation by orbits, and the leaves
on which the norm has a non-degenerate minimum are exactly the Siegel
leaves of Camacho, Kuiper and Palis \cite{CKP} and L\'opez de Medrano
\cite{LdM-Siegel}. The LVM manifold is the projectivized manifold of
tangencies, the quotient of the tangency locus by the canonical action
of $\C^*$: a compact real submanifold of $\PP^{n-1}_{\mathbb C}$ which meets each
leaf of the projectivized foliation exactly once, and therefore
inherits the transverse complex structure of the foliation.

The construction is remarkable for its simplicity: a single datum of
linear algebra produces compact complex manifolds, usually
non-K\"ahler, whose topology is governed by convex polytopes
\cite{MV,BM,Verjovsky-survey}. It was extended by Bosio and Meersseman
to the LVMB manifolds \cite{Bosio,BM}, and related to geometric
invariant theory and toric geometry by Meersseman--Verjovsky \cite{MV},
Cupit-Foutou--Zaffran \cite{CFZ}, Battisti \cite{Battisti} and others;
its irrational versions underlie the quantum toric geometry of
\cite{KLMV1,KLMV2}. Ishida \cite{Ishida} classified the compact complex
manifolds admitting a maximal torus action, a class which contains the
LVM and LVMB manifolds. In all these developments the acting group is
abelian.

It is therefore natural to ask whether the diagonal group $\C^k$ can be replaced
by a non-abelian group of triangular matrices. In general the answer is
negative, because of a commutator term. Write an element of the Lie
algebra of the acting group as $B=H_B+K_B$, with $H_B$ Hermitian (the
dilation part) and $K_B$ skew-Hermitian (the rotation part). The second
derivative of $\|e^{tB}z\|^2$ at $t=0$ is
$$
4\|H_Bz\|^2+2\langle[H_B,K_B]z,z\rangle .
$$
In the diagonal case $H_B$ and $K_B$ commute, so only the non-negative
first term remains, at every point of every leaf, and the whole LVM
construction rests on this convexity. If $[H_B,K_B]\neq0$, it is a
Hermitian matrix of trace zero, so the second term takes both signs and
may dominate the first: tangencies can be saddles or degenerate, and
the tangency locus need not be a transversal to the foliation. To our
knowledge, the LVM construction has so far been confined to abelian
groups, and this computation shows that the restriction is genuine.

In this paper we show that the obstruction disappears for a large class of
non-abelian solvable groups, and we explain why. Let $K\subset U(n)$ be
a compact connected group of even dimension $2k$, with complexification
$S\subset GL(n,\C)$. A \emph{Samelson group} is a complex subgroup
$G\subset S$ complementary to $K$: every element of $S$ is in a unique
way a product $gu$ with $g\in G$ and $u\in K$. These are the groups
used by Samelson \cite{Sam} and Wang \cite{Wang} to put left-invariant
complex structures on compact Lie groups of even dimension, and they
were classified by Pittie \cite{Pittie}. For such groups the tangency
locus of the $G$-orbit foliation is the zero set of the moment map of
$K$, and the commutator term vanishes there. It does not vanish
elsewhere, so the global argument is carried out along the Hermitian
one-parameter subgroups of $S$, which have no rotation part
(Remark~\ref{rem:commutator}); this convexity argument is the
Kempf--Ness theorem \cite{KN} in disguise. The ingredients are
classical. What is new is that they fit into the tangency picture of
L\'opez de Medrano, Verjovsky and Meersseman, which thereby acquires a
non-abelian theory; it is the simplicity of their original construction
that makes this extension possible.

Let $H^*$ be the zero set of the moment map of $K$ on $\C^n$, with the
origin removed. We say that $K$ is \emph{admissible} if $H^*$ is
non-empty and $K$ acts on it with finite stabilizers
(Definition~\ref{def:admissible}); for tori this is exactly
Meersseman's condition. Put $\cS=G\cdot H^*$, an open subset of $\C^n$,
and call
$$
N=N_G=H^*/\C^*
$$
the \emph{non-abelian LVM manifold} of $(K,G)$
(Definition~\ref{def:nonabelianLVM}), endowed with the complex
structure transverse to the projectivized foliation by $G$-orbits.

\begin{theoremA}
Let $K$ be admissible and let $G$ be a Samelson group in its
complexification. Then:
\begin{enumerate}
\item $H^*$ is the tangency locus of the orbits of $G$ with the spheres,
every tangency is a leafwise non-degenerate minimum, and every orbit in
$\cS$ meets $H^*$ exactly once;
\item $N$ is a compact complex manifold of dimension $n-k-1$,
biholomorphic to the space of leaves of the projectivized foliation;
\item $N$ fibres holomorphically, in the orbifold sense, over
$\PP(\cS)/S$, with model fibre $K$ endowed with its Samelson complex
structure; geometric fibres may be finite quotients of $K$.
If $K$ is not abelian, $N$ is not K\"ahler.
\end{enumerate}
\end{theoremA}

Parts (1) and (2) are Theorem~\ref{thm:main}, and part (3) is
Theorem~\ref{thm:structure}. For a torus, Theorem~A recovers, after the
normalization of Section~\ref{subsec:condition-K}, the LVM manifolds
satisfying the rationality condition (K) of Meersseman and Verjovsky
\cite{MV}. For $K=U(2)$ it gives homogeneous Hopf surfaces and complex
structures on the Stiefel manifolds $V_2(\C^p)$ fibred by Hopf surfaces
over Grassmannians, as in Ornea and Piccinni \cite{OP}; for
$K=U(2)\times U(2)$ and suitable Samelson groups the fibres are, up to
finite quotients, products of an elliptic curve and a Calabi--Eckmann
manifold \cite{CE}. On the other hand, $SU(2)\times SU(2)$, whose
Samelson structures are the Calabi--Eckmann structures on
$S^3\times S^3$, has no admissible representation
(Proposition~\ref{prop:no-SU2SU2}). Thus LVM, Hopf, Calabi--Eckmann and
Samelson manifolds appear in a single picture.

By part (3), the manifolds of Theorem~A are holomorphic bundles over
projective orbifolds whose fibres are compact Lie groups with Samelson
structures. Such bundles were studied by Poddar and Thakur \cite{PT},
whose general construction also allows non-projective bases; see also
Loeb, Manjar\'in and Nicolau \cite{LMN}. What the present approach adds
is their realization as projectivized manifolds of tangencies of linear
foliations, in exactly the same terms as the LVM manifolds, and, in the
irrational case below, the construction of examples whose invariants
can be computed.

The next question is what happens when the rotation parts generate a
non-closed group, the situation of the LVM manifolds without condition
(K). The non-abelian construction extends to this setting as well
(Theorem~\ref{thm:irrational}).

\begin{theoremB}
Parts \textup{(1)} and \textup{(2)} of Theorem~A remain valid when the
compact group $K$ is replaced by the
group $K'\cdot\exp\mathfrak z$ generated by a compact connected group
$K'$ and a possibly non-closed group of diagonal unitary matrices
commuting with $K'$, with $\dim(\mathfrak k'\oplus\mathfrak z)$ even,
and $G$ by a Samelson group in the sense of
Section~\ref{sec:irrational}, provided the tangency locus is non-empty
and the infinitesimal stabilizers vanish on it.
\end{theoremB}

For a torus, Theorem~B gives the LVM manifolds without condition (K).
In the simplest non-abelian irrational examples, built from binary
forms (Section~\ref{sec:binary-forms}), the irrationality is absorbed
into the moduli of the fibre, and one obtains bundles of Hopf manifolds
over projective orbifolds, of known types. Section~\ref{sec:absorption}
explains this absorption and identifies a limitation of the argument
when the irrational torus separates copies of a non-abelian
representation that can vanish independently. The examples below
exploit this situation.

An explicit situation in which this absorption argument does not apply is the following. Let $\zeta=e^{2\pi i/5}$
and $V=(\C^2)^5\oplus\C$, with coordinates $(v_0,\ldots,v_4,c)$. The
group $SU(2)$ acts on each $v_j$ by its standard representation and
trivially on $c$, and a three-dimensional non-closed group of diagonal
unitary matrices acts on the five copies of $\C^2$ through the vertices
$\zeta^j$ of a regular pentagon and on $c$ by a character. Together
with the Hermitian matrix $\diag(1,-1)$ and the nilpotent matrix
$\left(\begin{smallmatrix}0&1\\0&0\end{smallmatrix}\right)$ acting on
each $v_j$, this gives a Samelson group $G$ satisfying the hypotheses
of Theorem~B, and the resulting manifold has the explicit model
$$
N\cong\Bigl\{(v_0,\ldots,v_4)\in(\C^2)^5:\ \sum_jv_jv_j^*=\tfrac12I_2,\ \
\sum_j\zeta^j\|v_j\|^2=0\Bigr\}
$$
(Proposition~\ref{prop:N5}), with the complex structure transverse to
the orbits of $G$. Replacing the pentagon by $m\geq5$ suitable points
$\omega_j$ of the unit circle gives the manifolds $N_{\omega,\tau}$ of
Section~\ref{sec:families}; the pentagon is the case $m=5$.

\begin{theoremC}
For every $m\geq5$ there are compact complex manifolds $N$ of
dimension $2m-3$, obtained from Theorem~B, with the following
properties:
\begin{enumerate}
\item $N$ is $4$-connected and admits no symplectic form;
\item $N$ is a holomorphic bundle of Hopf surfaces over a compact
complex manifold of dimension $2m-5$, obtained from the stable part of
the Grassmannian $\Gr(2,m)$;
\item the meromorphic function field of $N$ is purely transcendental
of degree $m-2$, so that $N$ has algebraic dimension $m-2$, but every
holomorphic map from $N$ to a projective space is constant;
\item the identity component of the automorphism group of $N$ is
abelian of complex dimension $m-1$, and its maximal compact subgroup
is a torus of real dimension $m+1$.
\end{enumerate}
Consequently $N$ admits no maximal torus action, is not homogeneous,
and is not a Samelson-fibre bundle over a projective orbifold. For
$m=5$ this is the pentagon example above, of dimension $7$ and
algebraic dimension $3$.
\end{theoremC}

This is Theorem~\ref{thm:families} together with
Proposition~\ref{prop:seed}, which provides the required weights in
every $m\geq5$; for the pentagon see also Theorem~\ref{thm:N5},
Corollary~\ref{cor:new} and Proposition~\ref{prop:function-field}. By
(4), the maximal compact torus is too small to act maximally on a
manifold of dimension $2m-3$, so $N$ is not biholomorphic to any
manifold with a maximal torus action in the sense of Ishida, in
particular to no LVM or LVMB manifold and to no moment-angle manifold
with its standard complex structure. In (2) the base replaces the
projective spaces of the toric world by non-toric quotients of
Grassmannians: the non-abelian group is absorbed into the fibre, but it
changes the base. These bases have rational function fields of
transcendence degree $m-3$, generated by cross-ratios of the column
directions, and their first homotopy groups are $\pi_1=\pi_3=0$ and
$\pi_2=\pi_4=\Z$ (Theorem~\ref{thm:Z}). As Hopf-surface bundles,
the manifolds of Theorem~C fit the general construction of
Poddar--Thakur \cite[Theorem~3.3]{PT} over non-projective bases; what
distinguishes them is their realization as manifolds of tangencies,
their Grassmannian bases, and the fact that all the invariants above can
be computed. These families show that the non-abelian tangency
construction produces compact complex manifolds outside the class of
manifolds admitting maximal torus actions.

The Grassmannian bases are also a family of interest in their own
right. Theorem~\ref{thm:Z} computes their automorphism groups, homotopy
groups and meromorphic function fields. For $m=5$, the base $Z_5$ has
complex dimension five and algebraic dimension two, yet every
holomorphic map from it to a projective space is constant. Like the
total spaces, these bases admit no maximal torus action and, under
condition (R), no Samelson-fibre bundle presentation over a projective
orbifold.

The proofs of (3) and (4) lift holomorphic vector fields, meromorphic
functions and holomorphic maps to projective space to an affine model,
an open subset of $\C^{2\times m}$ whose complement has codimension at
least three. They use Hartogs and Scheja extension, the vanishing of
its Picard group, and the fact that one generator of the Samelson group
has all its eigenvalues in a half-plane. The $4$-connectedness comes
from the same codimension estimate, and implies the absence of
symplectic forms.

The paper is organized as follows. Section~\ref{sec:prelim} recalls LVM manifolds, Meersseman's
admissibility, the transverse complex structure of Haefliger, and
Samelson groups. Section~\ref{sec:tangency} studies the tangency locus,
and Section~\ref{sec:main} proves Theorem~A and describes the structure
of the manifolds. Section~\ref{sec:irrational} proves Theorem~B,
Section~\ref{sec:binary-forms} treats the binary forms, and
Section~\ref{sec:absorption} discusses absorption and open questions.
Section~\ref{sec:new} studies the pentagon example, and
Section~\ref{sec:families} proves Theorem~C in general and computes the
algebraic reductions. An elementary account of the construction will
also appear in a forthcoming book by the author, within the framework
of tangency geometry.

Throughout this paper $\langle\ ,\ \rangle$ is the
standard Hermitian product on $\C^n$, complex linear in the first
variable, and $\|z\|^2=\langle z,z\rangle$. Every matrix $B$ splits as
$B=H_B+K_B$ with $H_B=\tfrac12(B+B^*)$ Hermitian and
$K_B=\tfrac12(B-B^*)$ skew-Hermitian. We think of $e^{K_B}$ as a
\emph{rotation} and of $e^{H_B}$ as a \emph{dilation}.
Matrix brackets mean $[A,B]=AB-BA$. For the associated linear
vector fields $X_A(z)=Az$, regarded as derivations, we use the usual
bracket, so $[X_A,X_B]=-X_{[A,B]}$. When the same symbol denotes a
matrix and its fundamental field, the context specifies which bracket
is intended.

\section{Preliminaries}\label{sec:prelim}

\subsection{LVM manifolds and Meersseman's admissibility}

Let $\Lambda=(\Lambda_1,\ldots,\Lambda_n)$, $\Lambda_j\in\C^k$, with
$n>2k$, and let $\C^k$ act on $\C^n$ by
$$
T\cdot z=\bigl(z_1e^{\Lambda_1\cdot T},\ldots,
z_ne^{\Lambda_n\cdot T}\bigr),\qquad
\Lambda_j\cdot T=\sum_{\alpha=1}^k\Lambda_{j\alpha}T_\alpha.
$$

\begin{definition}[Meersseman \cite{M}]
\label{def:Meersseman-admissible}
The configuration $\Lambda$ satisfies the \emph{Siegel condition} if
$0\in\operatorname{Conv}\{\Lambda_1,\ldots,\Lambda_n\}$, and the
\emph{weak hyperbolicity condition} if
$0\notin\operatorname{Conv}\{\Lambda_{i_1},\ldots,\Lambda_{i_{2k}}\}$
for every choice of $2k$ indices. It is \emph{admissible} if it
satisfies both.
\end{definition}

For an admissible configuration let
$\cS_\Lambda=\{z:\ 0\in\operatorname{Conv}\{\Lambda_j:z_j\neq0\}\}$
be the union of the Siegel leaves \cite{CKP,LdM-Siegel}, and
$$
H_\Lambda=\Bigl\{z\in\C^n\setminus\{0\}:\
\sum_{j=1}^n\Lambda_j|z_j|^2=0\Bigr\}.
$$
The action is free on $\cS_\Lambda$, the restriction of $\|z\|^2$ to
each leaf has a unique critical point, a non-degenerate minimum, and
the set of these minima is $H_\Lambda$. Thus $H_\Lambda$ is the locus
where the leaves are tangent to the spheres, and it meets every leaf
in $\cS_\Lambda$ exactly once. The LVM manifold is
$N_\Lambda=H_\Lambda/\C^*\subset\PP^{n-1}_{\mathbb C}$; it is a compact complex
manifold of dimension $n-k-1$, not K\"ahler for $n>2k+1$
\cite{LdMV,M}.

\subsection{The classical rationality condition (K)}
\label{subsec:condition-K}

For an admissible configuration, the classical condition (K) of
Meersseman--Verjovsky \cite{MV} is that the real vector space
$$
\mathcal R_\Lambda=
\left\{s\in\R^n:\ \sum_{j=1}^n s_j\Lambda_j=0,\quad
\sum_{j=1}^n s_j=0\right\}
$$
has a basis of vectors in $\Z^n$. For this exact formulation see
\cite[\S5.4]{KLMV1}, where the associated torus-fibration
result of \cite{MV} is recalled as Theorem~5.1.

Equivalently, the augmented real row space generated by
$\mathbf1$, $\operatorname{Re}\Lambda$ and
$\operatorname{Im}\Lambda$ is rational. Thus (K) concerns the
rotation subgroup \emph{modulo scalar rotations}; it does not require
the affine rotation subgroup in $(S^1)^n$ itself to be closed.
It is invariant under real affine changes of the weights. In
particular, a common translation of the weights leaves the
projectivized holomorphic action unchanged.

To relate this formulation to integral weights for a compact torus,
choose a real affine coordinate system in which the weights are
integral. The image of the original origin can be replaced by a
nearby rational point in the same admissibility chamber. Here a
chamber means a region on which the subsets of weights whose convex
hulls contain that point are fixed. Translate this rational point
to zero and clear denominators. The resulting integral configuration
has the same stable projective open set and the same projectivized
holomorphic foliation: the real linear change is incorporated into
the choice of the Samelson subspace, and the translation only adds
scalar infinitesimal actions. This is the normalization used when
identifying the compact-torus construction with the rational LVM case.

For example, $(1,i,-\sqrt2-i)$ is an admissible configuration with
$\mathcal R_\Lambda=0$, so it satisfies (K), whereas its affine
rotation plane has irrational normal vector $(\sqrt2,1,1)$ and
does not exponentiate to a closed subgroup of $(S^1)^3$. Its image
modulo scalar rotations is closed.

\subsection{The transverse complex structure}

The complex structure of $N_\Lambda$ is the transverse structure of
the foliation, in the sense of Haefliger \cite{Haefliger}. We use it in
the following elementary form.

\begin{lemma}\label{lem:Haefliger}
Let $\cF$ be a regular holomorphic foliation of complex codimension
$q$ on a complex manifold $W$, and let $N\subset W$ be a smooth real
submanifold of dimension $2q$, transverse to $\cF$. Then $N$ has a
unique complex structure for which the restrictions to $N$ of the
local holomorphic submersions defining $\cF$ are local
biholomorphisms. If moreover every leaf meets $N$ exactly once, $N$ is
the space of leaves of $\cF$ with its transverse complex structure.
A biholomorphism of $W$ which preserves $\cF$ and $N$ restricts to a
biholomorphism of $N$.
\end{lemma}

\begin{proof}
By transversality the restriction to $N$ of a local submersion
$W\to\C^q$ defining $\cF$ is a local diffeomorphism. Two such
submersions differ by a local biholomorphism of $\C^q$, so the
restrictions form a holomorphic atlas. The other assertions follow
from the construction.
\end{proof}

\subsection{Samelson groups}\label{subsec:Samelson}

Let $K\subset U(n)$ be a compact connected Lie group with Lie algebra
$\mathfrak k$, a real space of skew-Hermitian matrices, of even
dimension $2k$. Its complexification $S\subset GL(n,\C)$ is the
connected complex Lie group with Lie algebra
$\mathfrak s=\mathfrak k\oplus i\mathfrak k$. It is closed under
$s\mapsto s^*$, and every $s\in S$ has a unique \emph{polar
decomposition} $s=u\,e^{P}$ with $u\in K$ and $P\in i\mathfrak k$
\cite[Ch.~VI]{Knapp}.

Let $T\subset K$ be a maximal torus, of dimension $r$, with Lie
algebra $\mathfrak t$, choose a system of positive roots of
$(\mathfrak s,\mathfrak t^{\C})$, and let $\mathfrak n^+$ be the sum
of the positive root spaces. Since
$\dim K=r+2\dim_\C\mathfrak n^+$, the rank $r$ is even.

\begin{definition}[\cite{Sam,Wang,Pittie}]\label{def:Samelson}
A \emph{Samelson subalgebra} of $\mathfrak s$ is a complex subalgebra
$\mathfrak g=\mathfrak a\oplus\mathfrak n^+$, where
$\mathfrak a\subset\mathfrak t^{\C}$ is a complex subspace of
dimension $r/2$ with $\mathfrak a\cap\mathfrak t=0$. The connected
subgroup $G\subset S$ with Lie algebra $\mathfrak g$ is a
\emph{Samelson group}. We write $A=\exp\mathfrak a$.
\end{definition}

The condition $\mathfrak a\cap\mathfrak t=0$ says that $\mathfrak a$
contains no infinitesimal rotation. A Samelson group is solvable, of
complex dimension $k$, and abelian exactly when $K$ is a torus; then
$G=A\cong\C^k$. For $K=U(r)$, $G$ is a group of upper triangular
matrices whose diagonal runs over the complex subgroup $A$ of the
diagonal torus. By a theorem of Pittie \cite{Pittie}, every complex
subalgebra of $\mathfrak s$ which is a real complement of
$\mathfrak k$ is a Samelson subalgebra for a suitable choice of
$\mathfrak t$ and of the positive roots.

\begin{lemma}\label{lem:Samelson}
Every element of $S$ is in a unique way a product $g\,u$ with
$g\in G$, $u\in K$. The group $G$ is closed in $S$, it is
biholomorphic to $\C^k$, and it has no non-trivial finite subgroups.
The complex torus $T^{\C}$ normalizes $G$, and $T^{\C}/A$ is a compact
complex torus of dimension $r/2$.
\end{lemma}

\begin{proof}
First, $\mathfrak g\cap\Ad(s)\mathfrak k=0$ for every $s\in S$. Let
$\xi=a+\nu$ be in the intersection, $a\in\mathfrak a$,
$\nu\in\mathfrak n^+$. The elements of $\Ad(s)\mathfrak k$ are
conjugate to skew-Hermitian matrices, hence diagonalizable with purely
imaginary eigenvalues. In a basis of common eigenvectors of $T$,
suitably ordered, $a$ is diagonal and $\nu$ strictly upper triangular,
so the eigenvalues of $\xi$ are the diagonal entries of $a$. Thus $a$
is skew-Hermitian, $a\in\mathfrak a\cap\mathfrak t=0$, and $\xi=\nu$
is nilpotent and diagonalizable, so $\xi=0$.

Let $G$ act on $S/K$, which by the polar decomposition is
diffeomorphic to $i\mathfrak k\cong\R^{2k}$. The isotropy algebra at
$sK$ is $\mathfrak g\cap\Ad(s)\mathfrak k=0$, and $\dim_\R G=2k$, so
all orbits are open; since $S/K$ is connected there is a single orbit,
and $S=GK$. The orbit map $G\to S/K$ is a covering of a simply
connected space, hence a diffeomorphism. This gives uniqueness, shows
that $G$ is closed and diffeomorphic to $\R^{2k}$, and, since a finite
subgroup of $S$ is conjugate into $K$ and therefore fixes a point of
$S/K$, that $G$ has no non-trivial finite subgroups. The group $G$ is
the product of $A\cong\C^{r/2}$ and of the unipotent group
$\exp\mathfrak n^+$, on which the exponential map is biholomorphic, so
$G\cong\C^k$. Finally $T^{\C}$ centralizes $\mathfrak a$ and preserves
$\mathfrak n^+$; and $\mathfrak t^{\C}=\mathfrak a\oplus\mathfrak t$
gives an isomorphism of real Lie groups $\mathfrak a\times T\to
T^{\C}$, $(a,t)\mapsto\exp(a)t$, so $T^{\C}/A\cong T$ is compact.
\end{proof}

Thus $K\cong G\backslash S$ inherits a complex structure, invariant
under right translations by $K$ (or left translations after identifying
the underlying group by inversion). This is the discovery of Samelson and
Wang \cite{Sam,Wang}: for $K=U(2)$ one obtains homogeneous Hopf surfaces, and
for $K=SU(2)\times SU(2)$ the Calabi--Eckmann manifolds \cite{CE}.

\section{The tangency locus}\label{sec:tangency}

Let $G$ be a Samelson group in the complexification of $K\subset
U(n)$, acting linearly on $\C^n$. For $X\in\mathfrak k$ put
$$
\mu_X(z)=\langle iXz,z\rangle ,
$$
a real quadratic form; $2\mu_X(z)$ is the derivative at $t=0$ of
$\|e^{itX}z\|^2$. The map $z\mapsto(X\mapsto\mu_X(z))$ is, up to a
constant, the moment map of the action of $K$ on $\C^n$. Let
$$
H^*=H^*_K=\{z\in\C^n\setminus\{0\}:\ \mu_X(z)=0\ \text{for all }
X\in\mathfrak k\},
$$
an intersection of $2k$ real quadrics, invariant under $K$ and under
the homotheties.

\begin{definition}\label{def:admissible}
The group $K\subset U(n)$ is \emph{admissible} if $H^*_K\neq\varnothing$
and $K$ acts on $H^*_K$ with finite stabilizers.
\end{definition}

The terminology is justified by the following proposition, which
shows that for tori this is Meersseman's admissibility
(Definition~\ref{def:Meersseman-admissible}).

\begin{proposition}\label{prop:tori}
Let $K=(S^1)^{2k}$ act diagonally on $\C^n$ with weights
$w_1,\ldots,w_n\in\mathfrak t^*\cong\R^{2k}$, let $G=A$ be a Samelson
group, and let $\Lambda_j\in\mathfrak a^*\cong\C^k$ be the restriction
to $\mathfrak a$ of the complex-linear extension of $w_j$. Then the
action of $G$ is the diagonal action with configuration
$\Lambda=(\Lambda_1,\ldots,\Lambda_n)$, $H^*_K=H_\Lambda$, and $K$ is
admissible if and only if $\Lambda$ is admissible in the sense of
Meersseman.
\end{proposition}

\begin{proof}
The first assertion is clear. Since $\mathfrak a\cap\mathfrak t=0$,
also $\mathfrak a\cap\overline{\mathfrak a}=0$, so
$\mathfrak t^{\C}=\mathfrak a\oplus\overline{\mathfrak a}$, and a real
weight vanishing on $\mathfrak a$ vanishes on $\overline{\mathfrak a}$.
Hence $w\mapsto w|_{\mathfrak a}$ is an isomorphism of real vector
spaces $\mathfrak t^*\to\mathfrak a^*$, and convex-hull conditions on
the $w_j$ and on the $\Lambda_j$ are equivalent. The quadrics are
$\mu_X(z)=-\sum_jw_j(X)|z_j|^2$, so $H^*_K=\{\sum_jw_j|z_j|^2=0\}=
H_\Lambda$, which is non-empty if and only if the Siegel condition
holds. The stabilizer of $z$ is finite if and only if the weights $w_j$
with $z_j\neq0$ span $\R^{2k}$. If at some $z\in H^*_K$ they did not,
they would lie in a hyperplane while containing $0$ in their convex
hull, and Carath\'eodory's theorem in dimension $2k-1$ would put $0$
in the convex hull of at most $2k$ of them, contradicting weak
hyperbolicity. Conversely, if $0$ is in the convex hull of $2k$
weights, these are linearly dependent, and a point of $H^*_K$
supported on them has an infinite stabilizer.
\end{proof}

\begin{proposition}\label{prop:tangency}
Let $z\in\C^n\setminus\{0\}$.
\begin{enumerate}
\item The orbit $G\cdot z$ is tangent at $z$ to the sphere through $z$
if and only if $z\in H^*$.
\item For $z\in H^*$ and $\xi\in\mathfrak g$,
$\dfrac{d^2}{dt^2}\Big|_{t=0}\|e^{t\xi}z\|^2=4\|H_\xi z\|^2$.
\item The map $\xi\mapsto H_\xi$ is a real isomorphism of $\mathfrak g$
onto $i\mathfrak k$. Hence, if the stabilizer $K_z$ is finite, the
tangency at $z\in H^*$ is a leafwise non-degenerate minimum, and near
$z$ the set $H^*$ is a smooth submanifold of real codimension $2k$
transverse to the orbits of $G$.
\end{enumerate}
\end{proposition}

\begin{proof}
(1) The orbit is tangent to the sphere at $z$ when
$\operatorname{Re}\langle\xi z,z\rangle=0$ for all $\xi\in\mathfrak g$.
Since $\langle Xz,z\rangle$ is purely imaginary for $X\in\mathfrak k$,
this holds for all $\xi\in\mathfrak g+\mathfrak k=\mathfrak s$; as
$\mathfrak s$ is complex we may replace $\xi$ by $i\xi$, and obtain
$\langle\xi z,z\rangle=0$ for all $\xi\in\mathfrak s$. Taking $\xi=iX$
gives $\mu_X(z)=0$. The converse is the same argument read backwards.
Here $\mathfrak g+\mathfrak k=\mathfrak s$ because
$\mathfrak g\cap\mathfrak k=0$ (Lemma~\ref{lem:Samelson}) and the
dimensions add up.

(2) For any matrix $B=H+K$ as in the conventions, a direct computation
gives
$$
\frac{d^2}{dt^2}\Big|_{t=0}\|e^{tB}z\|^2=4\|Hz\|^2+2\langle[H,K]z,z\rangle .
$$
For $B=\xi\in\mathfrak g\subset\mathfrak s$, both $H_\xi$ and $K_\xi$
belong to $\mathfrak s$, which is closed under adjoints. So
$[H_\xi,K_\xi]$ is a Hermitian element of $\mathfrak s$, that is,
$iX$ with $X\in\mathfrak k$, and the last term is $2\mu_X(z)=0$.

(3) If $H_\xi=0$ then $\xi\in\mathfrak g\cap\mathfrak k=0$, and both
spaces have real dimension $2k$. If $K_z$ is finite, no non-zero
element of $i\mathfrak k$ kills $z$, so $H_\xi z\neq0$ for $\xi\neq0$
and the Hessian in (2) is positive definite. The gradient of $\mu_X$
at $z$ is $2iXz$, so the differentials of the $2k$ quadrics are
independent at $z$, and $H^*$ is smooth of codimension $2k$ near $z$,
with $T_zH^*=\{v:\operatorname{Re}\langle iXz,v\rangle=0\ \forall
X\in\mathfrak k\}$. It is transverse to the orbit: if $\xi z\in T_zH^*$,
take $iX=H_\xi$; then, by the computation in (2),
$0=\operatorname{Re}\langle H_\xi z,\xi z\rangle=\|H_\xi z\|^2+
\tfrac12\langle[H_\xi,K_\xi]z,z\rangle=\|H_\xi z\|^2$, so $\xi=0$.
\end{proof}

\begin{remark}[The commutator term]\label{rem:commutator}
The proof of (2) shows what makes Samelson groups special. For a
general linear action the second derivative contains the term
$\langle[H,K]z,z\rangle=2\operatorname{Re}\langle Hz,Kz\rangle$,
which measures the interaction of the dilation and rotation parts.
Since $[H,K]$ is Hermitian of trace zero, this term has no sign unless
it vanishes identically, and it produces saddles and degenerate
tangencies; here it vanishes at the tangency points, because the
commutator is again one of the quadrics defining them. In the diagonal
case $[H,K]=0$, so $t\mapsto\|e^{t\xi}z\|^2$ is convex for every $z$,
which gives at once the uniqueness of the critical point on each
Siegel leaf. For a non-abelian Samelson group the commutator term only
vanishes on $H^*$, and off $H^*$ the norm need not be convex along the
orbits. For instance, if $K$ contains $U(3)$ acting on a summand $\C^3$
and the positive roots are chosen so that $\mathfrak n^+$ contains
$\nu=E_{12}+E_{23}$, then for $z=(2,0,-1)$ in that summand
$\|e^{t\nu}z\|^2=5-t^2+t^4/4$, which has a local maximum at $t=0$.
This is why uniqueness is proved in Proposition~\ref{prop:once} by
working in $S$ along the Hermitian one-parameter groups $e^{tP}$,
$P\in i\mathfrak k$, for which $K_P=0$ and there is no commutator term
(Lemma~\ref{lem:convexity}).
\end{remark}

\section{Non-abelian LVM manifolds}\label{sec:main}

\begin{lemma}[Convexity]\label{lem:convexity}
Let $P$ be Hermitian and $z\in\C^n$. Then $f(t)=\|e^{tP}z\|^2$ is
convex, with $f''(t)=4\|Pe^{tP}z\|^2$, and if $f$ has two distinct
critical points then $Pz=0$.
\end{lemma}

\begin{proof}
$f(t)=\langle e^{2tP}z,z\rangle$, so
$f'(t)=2\langle Pe^{tP}z,e^{tP}z\rangle$ and
$f''(t)=4\|Pe^{tP}z\|^2$. A convex function with two critical points
is constant between them, so $Pe^{tP}z=0$ for some $t$, hence $Pz=0$.
\end{proof}

\begin{proposition}\label{prop:once}
Let $z_0\in H^*$ and $s\in S$.
\begin{enumerate}
\item If $sz_0\in H^*$, then $sz_0\in Kz_0$.
\item The points of minimal norm in $S\cdot z_0$ are exactly those of
$K\cdot z_0$.
\item If $K_{z_0}$ is finite, the stabilizer of $z_0$ in $S$ is
$K_{z_0}$.
\end{enumerate}
\end{proposition}

\begin{proof}
Write $s=ue^P$ and $f(t)=\|e^{tP}z_0\|^2$. Since $z_0\in H^*$,
$f'(0)=2\langle Pz_0,z_0\rangle=0$, so $0$ is a global minimum of $f$
and $\|sz_0\|\geq\|z_0\|$. If $sz_0\in H^*$, then $w=e^Pz_0=u^{-1}sz_0\in
H^*$ and $f'(1)=2\langle Pw,w\rangle=0$; if $sz_0$ has minimal norm,
then $f(1)=f(0)$. In both cases $t=1$ is a second critical point, so
$Pz_0=0$ and $sz_0=uz_0$. This proves (1) and (2). For (3), if
$sz_0=z_0$ then $Pz_0=0$; writing $P=iX$ with $X\in\mathfrak k$, the
finiteness of $K_{z_0}$ gives $X=0$, so $s=u\in K$.
\end{proof}

Statements (1) and (2) are part of the theorem of Kempf and Ness
\cite{KN}. Let
$$
\cS=\cS_K=S\cdot H^*=G\cdot H^*,
$$
the second equality because $S=GK$ and $K$ preserves $H^*$.

\begin{theorem}\label{thm:main}
Let $K\subset U(n)$ be admissible, of dimension $2k$, and let $G$ be a
Samelson group in its complexification.
\begin{enumerate}
\item The map $\Phi:G\times H^*\to\C^n$, $(g,z)\mapsto gz$, is a
diffeomorphism onto the open set $\cS$. Hence $G$ acts freely on
$\cS$, and its orbits are the leaves, closed in $\cS$ and
biholomorphic to $\C^k$,
of a regular holomorphic foliation $\cF_G$.
\item The restriction of $\|z\|^2$ to each leaf of $\cF_G$ has a unique
critical point, a non-degenerate minimum, and the set of these minima
is $H^*$. Thus $H^*$ is the tangency locus of $\cF_G$ with the spheres,
and it is a global transversal meeting every leaf exactly once.
\item The quotient $N_G=H^*/\C^*=(H^*\cap S^{2n-1})/S^1$ is a compact
smooth manifold of dimension $2(n-k-1)$, transverse in $\PP^{n-1}_{\mathbb C}$ to
the projectivized foliation; with the structure of
Lemma~\ref{lem:Haefliger} it is a compact complex manifold of
dimension $n-k-1$, biholomorphic to the space of leaves of the
projectivized foliation.
\end{enumerate}
\end{theorem}

\begin{proof}
(1) By Proposition~\ref{prop:tangency}(3), $H^*$ is a manifold of
dimension $2n-2k$ transverse to the orbits, so $d\Phi_{(e,z)}$,
$(\xi,v)\mapsto\xi z+v$, is an isomorphism; by equivariance $\Phi$ is a
local diffeomorphism, and $\cS$ is open. If $gz=g'z'$ with
$z,z'\in H^*$, then $z'=hz$ with $h=g'^{-1}g$, and by
Proposition~\ref{prop:once}(1) $z'=uz$ for some $u\in K$; then
$u^{-1}h$ fixes $z$, so $u^{-1}h\in K$ by
Proposition~\ref{prop:once}(3), and $h\in G\cap K=\{e\}$ by
Lemma~\ref{lem:Samelson}. So $\Phi$ is injective, hence a
diffeomorphism, and in these coordinates $G$ acts by left translations
on the first factor.

(2) By Proposition~\ref{prop:tangency}(1) the critical points of the
norm on a leaf are its points in $H^*$; by (1) there is exactly one,
and it is a non-degenerate minimum by
Proposition~\ref{prop:tangency}(3).

(3) The homotheties commute with $G$, preserve $H^*$ and act freely on
it. For $z\in H^*$, (1) gives $\C^n=T_zH^*\oplus\mathfrak g z$, and
$\C z\subset T_zH^*$; passing to $\PP^{n-1}_{\mathbb C}$, the image of $H^*$ is a
transversal to the projectivized foliation, which is regular of
dimension $k$ near it. If $[gz]=[z']$ with $z,z'\in H^*$, then
$gz=\lambda z'\in H^*$, and (1) gives $g=e$; so every leaf of the
projectivized foliation of $\PP(\cS)$ meets $N_G$ exactly once, and
Lemma~\ref{lem:Haefliger} applies. Compactness is clear, since
$H^*\cap S^{2n-1}$ is closed in the sphere.
\end{proof}

\begin{definition}\label{def:nonabelianLVM}
Let $K\subset U(n)$ be an admissible compact connected group of
dimension $2k$, with complexification $S$, and let $G\subset S$ be a
Samelson group (Definition~\ref{def:Samelson}). The
\emph{non-abelian LVM manifold} of $(K,G)$ is the compact complex
manifold
$$
N_G=H^*/\C^*=(H^*\cap S^{2n-1})/S^1
$$
of dimension $n-k-1$ given by Theorem~\ref{thm:main}(3), endowed with
the complex structure transverse to the projectivized foliation
$\cF_G$ (Lemma~\ref{lem:Haefliger}).
\end{definition}

As a smooth manifold, $N_G$ depends only on $K$ and the representation;
its complex structure depends on the choice of $G$, that is, of
$\mathfrak a$. When $K$ is a torus, $G=A$ acts diagonally, and $N_G$ is
the LVM manifold of the configuration described in
Proposition~\ref{prop:tori}.

\begin{remark}
By the theorem of Kempf and Ness \cite{KN}, $\cS$ is the complement of
the null cone of $S$, the set of points whose $S$-orbit has $0$ in its
closure, and for admissible $K$ every semistable point is stable:
if the orbit of a semistable point $x$ were not closed, its closure
would contain a closed orbit $S\cdot y$, $y\in H^*$, of strictly
smaller dimension, and the stabilizer of $y$ would be infinite. Thus
$\PP(\cS)$ is the set of stable points of $\PP^{n-1}_{\mathbb C}$ for the action
of $S$ in the sense of geometric invariant theory \cite{MFK}.
\end{remark}

\subsection{Structure}

The group $K\times S^1$ acts on $H^*$ with finite stabilizers: if
$Xz=i\theta z$ with $X\in\mathfrak k$ and $z\in H^*$, then
$0=\mu_X(z)=-\theta\|z\|^2$, so $\theta=0$ and $X=0$. Let
$$
X_K=H^*/(K\times\C^*)=\PP(\cS)/S,
$$
the geometric quotient of the stable points; it is a projective
variety with quotient singularities, of dimension $n-1-2k$ \cite{MFK}.

\begin{theorem}\label{thm:structure}
Let $N_G$ be a non-abelian LVM manifold.
\begin{enumerate}
\item The projection $N_G\to X_K$ is holomorphic; it is a holomorphic
fibre bundle in the orbifold sense whose fibres are the compact Lie
group $K$ with the Samelson complex structure $G\backslash S$, or
quotients of it by finite groups.
\item The compact complex torus $T^{\C}/A$ acts holomorphically on
$N_G$ with finite stabilizers, and the quotient is a holomorphic fibre bundle in the orbifold sense over $X_K$,
with model fibre the flag manifold $K/T$; its geometric fibres may
be finite quotients of this flag manifold.
\item The elements of $K$ which normalize $G$, among them those of $T$,
act on $N_G$ by biholomorphisms.
\item If $K$ is not abelian, $N_G$ is not K\"ahler.
\end{enumerate}
\end{theorem}

\begin{proof}
(1) Under the identification of $N_G$ with the space of leaves, the
fibre through the class of $z$ is the space of leaves contained in the
orbit of $S\times\C^*$ through $z$, which is $G\backslash S/\Gamma_{[z]}$, where
$$
\Gamma_{[z]}=\{s\in S:sz\in\C^*z\}
$$
is the finite projective stabilizer. Its finiteness follows from
that of the stabilizer in $S\times\C^*$: projection to $S$ identifies
the latter with $\Gamma_{[z]}$, since $z\neq0$. This group need not
coincide with the affine stabilizer $S_z$. By Lemma~\ref{lem:Samelson},
$G\backslash S\cong K$. Local triviality follows from a local slice for
the action of $S\times\C^*$ on $\cS$, which is proper: by
Proposition~\ref{prop:once} the natural map
$(S\times\C^*)\times_{K\times S^1}(H^*\cap S^{2n-1})\to\cS$ is a
bijective local diffeomorphism, hence a diffeomorphism, and a Lie
group acts properly on a bundle associated to it with compact
structure group.
(2) $T^{\C}$ normalizes $G$ and $A\subset G$, so $T^{\C}/A$ acts on the
space of leaves, with finite stabilizers because those of
$S\times\C^*$ are finite. Dividing by it amounts to dividing
$\PP(\cS)$ by the Borel subgroup generated by $G$ and $T^{\C}$, and
$B\backslash S$ is biholomorphic, by inversion, to $S/B\cong K/T$.
Over an orbifold chart this is the model fibre; at a coarse-space
point the residual finite stabilizer may give a finite quotient.
(3) Such an element preserves the foliation and $H^*$; apply
Lemma~\ref{lem:Haefliger}.
(4) By (1), $N_G$ contains compact complex submanifolds which are
finite quotients of $K$ with a left-invariant complex structure. If
$N_G$ were K\"ahler, so would these fibres and their finite covers
$K$. By the theorem of Borel and
Remmert \cite{BR} a compact homogeneous K\"ahler manifold is the
product of a complex torus and a flag manifold. A flag manifold which
is not a point has non-zero $\pi_2$, while $\pi_2$ of a Lie group
vanishes; so $K$ would be diffeomorphic to a torus. This is impossible
for non-abelian $K$, since a compact connected Lie group with
non-trivial semisimple part has $\pi_3\neq0$.
\end{proof}

\subsection{Examples}

\begin{example}[Tori]
Let $K=(S^1)^{2k}$ act diagonally with integral weights. By
Proposition~\ref{prop:tori}, $N_G$ is the LVM manifold $N_\Lambda$ of
the configuration $\Lambda_j=w_j|_{\mathfrak a}$. They satisfy condition (K), since the augmented matrix of integral
weights has rational kernel. Conversely, a configuration satisfying
(K) has such a realization after the projective normalization in
Section~\ref{subsec:condition-K}. Thus Theorem~\ref{thm:structure}(1)
recovers their orbifold fibration in compact complex tori over a
projective toric variety, as in Meersseman--Verjovsky \cite{MV}.
\end{example}

\begin{example}[Hopf surfaces and Stiefel manifolds]\label{ex:U2}
Let $K=U(2)$ act on pairs $(M,c)$, with $M$ a complex $2\times p$
matrix, $p\geq2$, and $c\in\C$, by $u\cdot(M,c)=(uM,(\det u)^{-m}c)$
with $m\geq1$. Then $n=2p+1$, $k=2$, and one computes
$\mu_X(M,c)=\operatorname{tr}\bigl(iX(MM^*-m|c|^2I)\bigr)$, so
$$
H^*=\{(M,c):\ MM^*=m|c|^2I,\ c\neq0\},
$$
on which $U(2)$ acts freely. Normalizing $c=1$ shows that $N_G$ is
diffeomorphic to the Stiefel manifold $V_2(\C^p)$ of unitary
$2$-frames, and Theorem~\ref{thm:structure} exhibits it as a bundle of
Hopf surfaces over the Grassmannian of planes in $\C^p$; for $p=2$,
$N_G$ is a Hopf surface. Here $G$ is the group of matrices
$\left(\begin{smallmatrix}e^{\alpha s}&b\\0&e^{\beta s}\end{smallmatrix}
\right)$, $s,b\in\C$, with $\alpha,\beta$ linearly independent over
$\R$, and the existence and uniqueness of the tangency point on each
orbit is the Gram--Schmidt process. Replacing $c$ by several
coordinates $c_1,\ldots,c_q$, on which $u$ acts by $(\det u)^{-m_j}$,
gives Hopf surface bundles over weighted projective bundles over
Grassmannians.
\end{example}

The Samelson structures on $SU(2)\times SU(2)\cong S^3\times S^3$ are
Calabi--Eckmann structures \cite{CE}. They do not occur directly as
fibres in Theorem~\ref{thm:structure}, because this group has no
admissible representation.

\begin{proposition}\label{prop:no-SU2SU2}
No unitary representation of $K=SU(2)\times SU(2)$ is admissible.
\end{proposition}

\begin{proof}
Let $T$ be the diagonal maximal torus, with weights in $\Z^2$, and
roots $\pm(2,0)$, $\pm(0,2)$. For $x\in\C^n$, $\mu(x)=0$ if and only
if $\langle Hx,x\rangle=0$ for $H\in\mathfrak t^{\C}$ and
$\langle E_\alpha x,x\rangle=0$ for every root vector $E_\alpha$. If $x$
is a sum of weight vectors no two of which have weights differing by a
root, the second condition holds, since $E_\alpha$ raises weights by
$\alpha$.

The set of weights is invariant under the Weyl group, which contains
$-1$. Suppose first that some weight $w$ is not one of $\pm(1,0)$,
$\pm(0,1)$. If $w=0$, a weight vector $x$ of weight $0$ satisfies
$\mu(x)=0$ and is fixed by $T$. If $w\neq0$, let $e_w,e_{-w}$ be unit
weight vectors and $x=e_w+e_{-w}$. The torus component of $\mu(x)$ is
proportional to $w-w=0$, and $2w$ is not a root, so $\mu(x)=0$; and $x$
is fixed by the circle of $T$ orthogonal to $w$. In both cases $H^*$
contains a point with infinite stabilizer.

Otherwise $\C^n\cong(\C^2\otimes\C)^{a}\oplus(\C\otimes\C^2)^{b}$, that
is, $K$ acts on $\C^{2\times a}\times\C^{2\times b}$ by
$(M_1,M_2)\mapsto(u_1M_1,u_2M_2)$, and $H^*$ is defined by
$M_1M_1^*=s I_2$, $M_2M_2^*=t I_2$ with $s,t\ge0$, not both zero. If
$a\ge2$, the points with $M_1M_1^*=I_2$, $M_2=0$ lie in $H^*$ and are
fixed by the second factor. If $a\le1$, then $M_1=0$ on $H^*$, and
every point of $H^*$ is fixed by the first factor; $H^*$ is empty if
also $b\le1$. So $K$ is never admissible.
\end{proof}

Calabi--Eckmann manifolds do appear as fibres once central directions
are added, which couple the two factors.

\begin{example}[Calabi--Eckmann fibres]\label{ex:CE}
Let $K=U(2)\times U(2)$ act on $\C^{2\times p}\times\C^{2\times q}\times
\C$, $p,q\ge2$, by
$$
(u_1,u_2)\cdot(M_1,M_2,c)=\bigl(u_1M_1,\,u_2M_2,\,
(\det u_1)^{-m_1}(\det u_2)^{-m_2}c\bigr),\qquad m_1,m_2\ge1 .
$$
Then $n=2p+2q+1$, $k=4$, and
$$
H^*=\{M_1M_1^*=m_1|c|^2I_2,\ M_2M_2^*=m_2|c|^2I_2,\ c\neq0\},
$$
on which $K$ acts freely, since both $M_j$ have rank $2$. Normalizing
$c=1$, $N_G$ is diffeomorphic to $V_2(\C^p)\times V_2(\C^q)$, and
Theorem~\ref{thm:structure} exhibits it as a bundle over
$\Gr(2,p)\times\Gr(2,q)$ with fibre $G\backslash S$, up to a finite
group, where $S=GL(2,\C)\times GL(2,\C)$. Let $Z_j$ and
$H_j=\diag(1,-1)$ denote the identity and the Cartan element of the
$j$-th factor, and $E_j$ its positive root vector. For
$\tau,\sigma\in\C\setminus\R$,
$$
\mathfrak g=\C(Z_1+\tau Z_2)\oplus\C(H_1+\sigma H_2)\oplus\C E_1
\oplus\C E_2
$$
is a Samelson subalgebra, since a vector
$x(Z_1+\tau Z_2)+y(H_1+\sigma H_2)$ with purely imaginary coordinates
has $x=y=0$. Then $G=G_Z\cdot G_{ss}$, with $G_Z=\exp\C(Z_1+\tau Z_2)$
central and $G_{ss}\subset SL(2,\C)^2$, and since
$S=\bigl((\C^*)^2\times SL(2,\C)^2\bigr)/(\pm1,\pm1)$,
$$
G\backslash S\cong\bigl(E_\tau\times\mathcal C_\sigma\bigr)/(\Z/2)^2,
$$
where
$$
E_\tau=(\C^*)^2/\{(e^t,e^{\tau t}):t\in\C\},\qquad
\mathcal C_\sigma=G_{ss}\backslash SL(2,\C)^2.
$$
The first factor is an elliptic curve. Identifying
$N^+\backslash SL(2,\C)$ with $\C^2\setminus\{0\}$ through the second
row expresses the second factor as
$$
\mathcal C_\sigma=(\C^2\setminus\{0\})^2/\C,\qquad
t\cdot(w_1,w_2)=(e^{-t}w_1,e^{-\sigma t}w_2).
$$
This is the Calabi--Eckmann manifold $S^3\times S^3$ \cite{CE}.
Thus the fibres are, up to finite
quotients, products of an elliptic curve and a Calabi--Eckmann
manifold. If instead $\mathfrak a=\C\diag(\alpha_1,\beta_1)\oplus
\C\diag(\alpha_2,\beta_2)$ splits along the factors, the fibres are,
up to finite quotients, products of two Hopf surfaces.
\end{example}

\section{The irrational case}\label{sec:irrational}

In the LVM construction the rotation parts generate a possibly
non-closed subgroup of a compact torus. Closedness of its image
modulo scalar rotations is the classical condition (K), recalled in
Section~\ref{subsec:condition-K}; affine closedness is a stronger
requirement for a fixed configuration. The non-abelian analogue is
as follows.

Let $K'\subset U(n)$ be a compact connected Lie group, and let
$\mathfrak z\subset\mathfrak u(n)$ be a real space of diagonal
skew-Hermitian matrices commuting with $K'$, such that
$\mathfrak k=\mathfrak k'\oplus\mathfrak z$ has even dimension $2k$.
We do not assume that $\exp\mathfrak z$ is closed in $U(n)$. To avoid
topological difficulties we work with abstract groups: let
$$
\widetilde K=K'\times\mathfrak z,\qquad
\widetilde S=K'^{\C}\times\mathfrak z^{\C},
$$
where $\mathfrak z$ and $\mathfrak z^{\C}=\mathfrak z\oplus i\mathfrak z$
are regarded as additive groups, and let $\rho:\widetilde S\to GL(n,\C)$,
$\rho(s,w)=s\,e^{w}$, which is a holomorphic homomorphism because
$\mathfrak z$ commutes with $K'$. Then $\widetilde K$ is a closed
subgroup of $\widetilde S$, and the polar decomposition
$\widetilde S=\widetilde K\exp(i\mathfrak k)$ holds factorwise. Let
$\mathfrak t'$ be the Lie algebra of a maximal torus of $K'$, and let
$\mathfrak a\subset(\mathfrak t'\oplus\mathfrak z)^{\C}$ be a complex
subspace of half the real dimension of $\mathfrak t'\oplus\mathfrak z$
with $\mathfrak a\cap(\mathfrak t'\oplus\mathfrak z)=0$. The
\emph{Samelson group} is the connected subgroup $\widetilde G\subset
\widetilde S$ with Lie algebra $\mathfrak g=\mathfrak a\oplus
\mathfrak n^+$; it acts on $\C^n$ through $\rho$. We define $\mu_X$,
$X\in\mathfrak k$, and $H^*$ as before. ``Irrational'' means that the
image of $\exp\mathfrak z$ in the torus of $\PP^{n-1}_{\mathbb C}$, that is,
modulo the scalars, is not closed.

\begin{theorem}\label{thm:irrational}
Assume that $H^*\neq\varnothing$ and that the infinitesimal stabilizer
$$\mathfrak k_z=\{X\in\mathfrak k:\ Xz=0\}$$ vanishes for every $z\in H^*$.
Then the map $\widetilde G\times H^*\to\C^n$, $(g,z)\mapsto\rho(g)z$, is
a diffeomorphism onto an open set $\cS$; the group $\widetilde G$ acts
freely on $\cS$, every tangency of its orbits with the spheres is a
leafwise non-degenerate minimum, and every orbit meets $H^*$ exactly
once; and $N=H^*/\C^*=\PP(\cS)/\widetilde G$ is a compact complex
manifold of dimension $n-k-1$, the space of leaves of the
projectivized foliation, with its transverse complex structure.
\end{theorem}

\begin{proof}
We check that each step of Sections~\ref{sec:tangency}
and~\ref{sec:main} goes through with $\widetilde S$, $\widetilde K$,
$\widetilde G$ in place of $S$, $K$, $G$. The Lie algebra $\mathfrak s=
\mathfrak k\oplus i\mathfrak k$ is still closed under adjoints, so
Proposition~\ref{prop:tangency} holds verbatim, with the hypothesis on
$K_z$ replaced by $\mathfrak k_z=0$. In the proof of
Lemma~\ref{lem:Samelson}, the eigenvalue argument uses only that
$\mathfrak a$ contains no infinitesimal rotation, and
$\widetilde S/\widetilde K\cong i\mathfrak k$ is still a vector space
on which $\widetilde G$ acts with open orbits; hence
$\widetilde S=\widetilde G\widetilde K$ with uniqueness, and
$\widetilde G\cap\widetilde K=\{e\}$. Lemma~\ref{lem:convexity} is
unchanged. In Proposition~\ref{prop:once}, statements (1) and (2) use
only the polar decomposition and the invariance of $H^*$ under
$\widetilde K$, and in (3) a stabilizing element $ue^{iX}$ gives
$Xz_0=0$, hence $X=0$ by $\mathfrak k_{z_0}=0$. The proof of
Theorem~\ref{thm:main} then applies word for word: injectivity of the
action map uses only these facts, the action of $\widetilde G$ on
$\cS$ is free, which also shows that $\rho$ is injective on
$\widetilde G$, and $H^*\cap S^{2n-1}$ is compact. Neither the
compactness of $\widetilde K$ nor the closedness of
$\rho(\widetilde G)$ is used.
\end{proof}

\begin{remark}[Principal bundles]\label{rem:principal-bundle}
The quotient maps $\cS\to\cS/\widetilde G$ and
$\PP(\cS)\to N$ are holomorphic principal $\widetilde G$-bundles.
Indeed, the global smooth product decomposition in
Theorem~\ref{thm:irrational} identifies the action with left
translation on the group factor, so it is proper as well as free.
Scalar multiplication acts only on the $H^*$ factor, by uniqueness
of the tangency representative; hence the projectivized action also
has a global smooth product decomposition. Local holomorphic
transverse charts and the holomorphic inverse function theorem give
local holomorphic sections and principal-bundle trivializations.
The same conclusion holds in Theorem~\ref{thm:main}.
\end{remark}

Since $\widetilde K$ is no longer compact, there is in general no
quotient of $N$ by it, and Theorem~\ref{thm:structure} has no direct
analogue. The question is what $N$ looks like.

\section{Binary forms with a Hopf part}\label{sec:binary-forms}

Let $V_d$ be the space of binary forms of odd degree $d\geq3$, with
the action of $SU(2)$ and its invariant Hermitian metric, let
$q\geq1$, and let $V=V_d\oplus\C^q$, with coordinates
$(f,c_1,\ldots,c_q)$. Let $K'=SU(2)$ act on $f$, and $\mathfrak z=\R\,
iY$ with
$$
Y=\diag(a_0I_{V_d},a_1,\ldots,a_q),\qquad a_0<0<a_1,\ldots,a_q .
$$
Then $\dim\mathfrak k=4$, $k=2$, and $\mathfrak t'\oplus\mathfrak z$ is
spanned by $iH$ and $iY$, where $H$ is the Cartan element of
$\mathfrak{sl}_2$ and $E$ the positive root vector. A Samelson
subalgebra is $\mathfrak g=\C(H+\tau Y)\oplus\C E$ with $\tau\notin\R$,
and we take $\operatorname{Im}\tau>0$. The image of $\exp\mathfrak z$
in the torus of $\PP(V)$ is the one-parameter group
$t\mapsto(e^{it(a_j-a_0)})_{j=1}^q$, which is not closed exactly when
some ratio $(a_i-a_0)/(a_j-a_0)$ is irrational; this requires
$q\geq2$.

\medskip\noindent\textbf{Tangency locus and admissibility.} Let $\mu'$
be the moment map of $SU(2)$ on $V_d$. Then
$$
H^*=\Bigl\{(f,c)\neq0:\ \mu'(f)=0,\ \
a_0\|f\|^2+\sum_ja_j|c_j|^2=0\Bigr\}.
$$
By the signs of the $a_j$, both $f$ and $c$ are non-zero on $H^*$. If
$(X,t)\in\mathfrak{su}(2)\oplus\R$ kills $(f,c)\in H^*$, then
$ta_jc_j=0$ for all $j$, so $t=0$, and $Xf=0$; a non-zero element of
$\mathfrak{su}(2)$ is conjugate to a multiple of $iH$, whose weights
on $V_d$ are the odd integers $d,d-2,\ldots,-d$, so $X=0$. Hence
Theorem~\ref{thm:irrational} applies, and
$$
N\cong(\Sigma_d\times S^{2q-1})/S^1,\qquad
\Sigma_d=\mu'^{-1}(0)\cap S^{2d+1},
$$
is a compact complex manifold of dimension $d+q-2$.

\medskip\noindent\textbf{The open set $\cS$.} We have
$\cS=V_d^s\times(\C^q\setminus\{0\})$, where $V_d^s$ is the set of
$SL_2$-stable binary forms; for odd $d$, stable and semistable forms
coincide \cite{MFK}. Indeed, by Kempf--Ness $SL_2$ moves a stable $f$ to
a form with $\mu'(f)=0$, and then the dilation $\exp(tY)$, which
commutes with $SL_2$, solves
$a_0e^{2a_0t}\|f\|^2+\sum_ja_je^{2a_jt}|c_j|^2=0$ in a unique way when
$c\neq0$, since the left side increases strictly from $-\infty$ to
$+\infty$. The converse inclusion is clear.

\medskip\noindent\textbf{The structure of $N$.} Let $B\subset SL_2$ be
the Borel subgroup of upper triangular matrices and
$$
Y_d=\PP(V_d^s)/B ,
$$
a compact complex orbifold of dimension $d-2$. It fibres over the
moduli space $\PP(V_d)/\!\!/SL_2$ of binary forms of degree $d$, of
dimension $d-3$, with fibre $SL_2/B=\PP^1_{\mathbb C}$ modulo finite groups. More explicitly,
it is the diagonal $SL_2$-quotient of
$\PP(V_d^s)\times\PP^1_{\mathbb C}$. With the linearization
$\mathcal O(a)\boxtimes\mathcal O(1)$ for sufficiently large $a$,
the stable and semistable loci are both
$\PP(V_d^s)\times\PP^1_{\mathbb C}$: the nonzero integral Hilbert--Mumford
weights of an odd-degree binary form dominate the bounded
contribution from $\PP^1_{\mathbb C}$. The resulting GIT quotient is projective
and has finite quotient singularities. This gives the projective
orbifold structure on $Y_d$.

\begin{proposition}\label{prop:example}
The map $[f:c]\mapsto B\cdot[f]$ induces a holomorphic map $N\to Y_d$,
which is a holomorphic fibre bundle in the orbifold sense. Its fibre
over a point with trivial stabilizer in $B/\{\pm I\}$ is the
diagonal Hopf manifold
$$
\mathcal H=(\C^q\setminus\{0\})/\langle\gamma\rangle,\qquad
\gamma=\diag\bigl(-e^{\pi i\tau(a_j-a_0)}\bigr)_{j=1}^q ,
$$
and over the other points a quotient of $\mathcal H$ by a finite
group. In local orbifold trivializations, the transition functions act on
$\mathcal H$ through the group generated by scalar homotheties and
$w\mapsto\diag(e^{w\tau(a_j-a_0)})_j$. Both actions commute with
$\gamma$.
\end{proposition}

\begin{proof}
The element $g(s,b)=\exp(s(H+\tau Y))\exp(bE)$ acts by
$$
(f,c)\longmapsto\bigl(e^{s\tau a_0}e^{sH}e^{bE}f,\
(e^{s\tau a_j}c_j)_j\bigr).
$$
On $[f]$ it acts through $e^{sH}e^{bE}$, and these elements fill $B$.
So $N=\PP(\cS)/\widetilde G\to\PP(V_d^s)/B$ is well defined and
holomorphic. Over the class of $f_0$, using $\widetilde G$ and the
homotheties we may assume $f=f_0$; the remaining freedom consists of
the pairs $(\lambda,g(s,b))$ with
$\lambda e^{s\tau a_0}e^{sH}e^{bE}f_0=f_0$. Then $e^{sH}e^{bE}$
stabilizes $[f_0]$. Suppose that the stabilizer of $[f_0]$ in $B$ is
$\{\pm I\}$, which acts trivially on $\PP(V_d)$. Then $b=0$ and
$s\in\pi i\Z$. Since the weights of $H$ on $V_d$ are odd, $e^{sH}$
acts on $V_d$ as the scalar $(-1)^{s/\pi i}$, and
$\lambda=(-1)^{s/\pi i}e^{-s\tau a_0}$. For $s=\pi i$ this element
acts on $c$ by $c_j\mapsto-e^{\pi i\tau(a_j-a_0)}c_j$, that is, by
$\gamma$. Hence the fibre is $(\C^q\setminus\{0\})/\langle\gamma\rangle$.
Since $a_j-a_0>0$ and $\operatorname{Im}\tau>0$, all eigenvalues of
$\gamma$ have modulus $<1$, so this is a Hopf manifold. A larger
(finite) stabilizer gives a finite quotient. Local triviality follows
from the same computation with $f_0$ varying in a local slice of
$\PP(V_d^s)\to Y_d$, and the transition functions also record changes of the representative
of $[f_0]$. Indeed, if two representatives are related by
$f_\beta=\nu e^{sH}e^{bE}f_\alpha$, the induced coordinate change is
$c_{\beta,j}=\nu e^{s\tau(a_j-a_0)}c_{\alpha,j}$.
The scalar factor $\nu$ cannot in general be omitted.
\end{proof}

Thus $N$ is a bundle of Hopf manifolds over a projective orbifold,
with transition functions generated by scalar and weighted diagonal
actions on the Hopf fibre. The irrational parameters enter the diagonal contraction defining the
Hopf fibre; this does not assert that every such contraction is realized
by rational weights. For $q=1$ the projected rotation group is a circle,
although its affine lift need not be closed unless the original weights
are commensurable. The fibre of $N\to Y_d$ is then an elliptic curve,
and the composite map to $\PP(V_d)/\!\!/SL_2$ has Hopf-surface fibres,
up to finite quotients. These examples therefore remain within the
class of fibrations over projective orbifolds considered here.

\section{Absorption, and where new manifolds may come from}
\label{sec:absorption}

The proof of Proposition~\ref{prop:example} uses only one feature:
\emph{on the part $V_d$ of the representation where $K'$ acts
non-trivially, the irrational torus acts by scalars}. After
projectivizing that part the torus disappears, $\widetilde G$ acts on
$\PP(V_d)$ through the Borel subgroup of $K'^{\C}$, and the torus
survives only on the part where $K'$ acts trivially, where it produces
an LVM-type manifold (here a Hopf manifold) as fibre.

\begin{remark}[Absorption]\label{rem:absorption}
The mechanism applies whenever projectivization removes the central
scalar action, the projected solvable group fills the relevant Borel
subgroup, and the projected points are stable with finite stabilizers.
Under these additional hypotheses the local-slice argument of
Proposition~\ref{prop:example} identifies the remaining fibres.
Scalar action on an isotypic component alone does not guarantee these
hypotheses, nor does it by itself prove local triviality of the quotient.
\end{remark}

This absorption argument need not apply when some $K'$-invariant component on which the
torus acts by a scalar can \emph{vanish} on $\cS$ while another one
does not. Then the corresponding projections are defined only on strata, so
the preceding argument does not produce a global bundle presentation.
This does not exclude a different absorption mechanism or another
fibration, as the Grassmannian construction below illustrates. This is
the non-abelian counterpart of the passage from Hopf and Calabi--Eckmann
manifolds to general LVM manifolds, where the combinatorics of a
polytope governs how the strata are glued.

\begin{example}[The smallest candidate]\label{ex:candidate}
Let $SU(2)$ act diagonally on two binary cubics, $V=V_3\oplus V_3\oplus
\C$, with coordinates $(f_1,f_2,c)$, and let $\mathfrak z=\R\,iY$,
$Y=\diag(a_1I,a_2I,a_3)$, with $a_3<0<a_1<a_2$ and
$(a_2-a_1)/(a_3-a_1)\notin\mathbb Q$. Then $k=2$,
$\mathfrak g=\C(H+\tau Y)\oplus\C E$, and
$$
H^*=\bigl\{\mu'(f_1)+\mu'(f_2)=0,\ \
a_1\|f_1\|^2+a_2\|f_2\|^2+a_3|c|^2=0\bigr\}\setminus\{0\}.
$$
On $H^*$ one has $c\neq0$ and $(f_1,f_2)\neq0$. If $(X,t)$ kills a
point of $H^*$ then $ta_3c=0$, so $t=0$, and $Xf_1=Xf_2=0$; since no
non-zero element of $\mathfrak{su}(2)$ kills a non-zero binary cubic,
$X=0$. By Theorem~\ref{thm:irrational}, $N$ is a compact complex
manifold of dimension $6$. The points with $f_2=0$, or with $f_1=0$,
occur on $H^*$ and form closed complex submanifolds of the type of
Section~\ref{sec:binary-forms} with $q=1$; and since $\mathfrak a$ mixes
the Cartan direction of $SU(2)$ with $Y$, which does not act by scalars
on $V_3\oplus V_3$, Remark~\ref{rem:absorption} does not apply. The
evident torus acting holomorphically on $N$ (the phases of
$f_1,f_2,c$ modulo scalars, and the maximal torus of $SU(2)$) has
dimension $3$, while $\dim_{\R}N=12$, so this action is not maximal in
the sense of Ishida \cite{Ishida}.
\end{example}

\begin{question}
Describe $N$ in Example~\ref{ex:candidate}: the open stratum
$\{f_1\neq0,\ f_2\neq0\}$, and the way the two closed strata are
glued to it. Does $N$ admit a holomorphic map with positive-dimensional
compact fibres onto a lower-dimensional space? What are its algebraic
dimension and its automorphism group? Does it admit a maximal torus
action?
\end{question}

If $N$ admits no such map and no maximal torus action, it is neither
an LVMB manifold \cite{Bosio,BM,Ishida}, nor a bundle with positive-dimensional Samelson
fibres, nor a bundle of the projective-base type considered above. This does
not exclude arbitrary presentations as quotients by abelian groups.
The methods of
the next section should apply to it. There we carry them out for a
different example, in which several copies of the same irreducible
representation of $SU(2)$ are separated by an irrational torus, and
compute invariants distinguishing it from the classes specified below.

\section{An explicit seven-dimensional example}\label{sec:new}

\subsection{The example}
Let $\zeta=e^{2\pi i/5}$ and $\lambda_j=(\cos(2\pi j/5),\sin(2\pi j/5))
\in\R^2$, $0\leq j\leq4$, the vertices of a regular pentagon. Let
$V=(\C^2)^5\oplus\C$, with coordinates $(v_0,\ldots,v_4,c)$, let
$K'=SU(2)$ act on each $v_j$ by its standard representation and
trivially on $c$, and let $\mathfrak z=\operatorname{span}_\R\{iY_0,
iY_1,iY_2\}$, where
$$
Y_0=\diag(I_2,\ldots,I_2,-1),\quad
Y_1=\diag\bigl((\operatorname{Re}\zeta^j)I_2;0\bigr),\quad
Y_2=\diag\bigl((\operatorname{Im}\zeta^j)I_2;0\bigr),
$$
the last entry being the action on $c$. Then $\dim\mathfrak k=6$,
$k=3$, $n=11$. With $H=\diag(1,-1)$ and
$E=\left(\begin{smallmatrix}0&1\\0&0\end{smallmatrix}\right)$ acting on
each $v_j$ and trivially on $c$, put
$$
D_0=H+iY_0,\qquad D_1=Y_1+iY_2,\qquad
\mathfrak g=\C D_0\oplus\C D_1\oplus\C E .
$$
Here $\mathfrak a=\C D_0\oplus\C D_1$; the Hermitian part of
$\alpha D_0+\beta D_1$ is $(\operatorname{Re}\alpha)H-(\operatorname{Im}
\alpha)Y_0+(\operatorname{Re}\beta)Y_1-(\operatorname{Im}\beta)Y_2$,
which vanishes only for $\alpha=\beta=0$, so
$\mathfrak a\cap(\mathfrak t'\oplus\mathfrak z)=0$ and $\mathfrak g$ is a
Samelson subalgebra in the sense of Section~\ref{sec:irrational}. The
brackets are $[D_0,E]=2E$ and $[D_0,D_1]=[D_1,E]=0$.

The integer vectors $a\in\Z^5$ with $\sum_ja_j\zeta^j=0$ are the
multiples of $(1,\ldots,1)$, since the minimal polynomial of $\zeta$ is
$1+x+x^2+x^3+x^4$. Hence the closure of the one-parameter groups
generated by $iY_1,iY_2$ in $(S^1)^5$ is $\{\prod_ju_j=1\}$, and the
image of $\exp\mathfrak z$ in the torus of $\PP(V)$ is not closed: the
example is irrational.

\begin{proposition}\label{prop:N5}
The tangency locus is
$$
H^*=\Bigl\{\sum_jv_jv_j^*=\tfrac12\Bigl(\sum_j\|v_j\|^2\Bigr)I_2,\ \
\sum_j\|v_j\|^2=|c|^2,\ \ \sum_j\lambda_j\|v_j\|^2=0\Bigr\}\setminus\{0\},
$$
the infinitesimal stabilizers vanish on it, and by
Theorem~\ref{thm:irrational}
$$
N=H^*/\C^*\cong\Bigl\{(v_j)\in(\C^2)^5:\ \sum_jv_jv_j^*=\tfrac12I_2,\ \
\sum_j\zeta^j\|v_j\|^2=0\Bigr\}
$$
is a compact complex manifold of dimension $7$.
\end{proposition}

\begin{proof}
The equations are the vanishing of the quadrics $\mu_X$ for $X$ in
$\mathfrak{su}(2)$ and for $iY_0,iY_1,iY_2$. The second one gives
$c\neq0$, and normalizing $c=1$ gives the description of $N$. Suppose
that $X+i(t_0Y_0+t_1Y_1+t_2Y_2)$, $X\in\mathfrak{su}(2)$, kills
$z\in H^*$. The action on $c$ gives $t_0=0$; put $t=(t_1,t_2)$. If
$X\neq0$, after conjugation $X=i\theta H$ with $\theta\neq0$, every
non-zero $v_j$ lies in an eigenline of $H$, and
$\varepsilon_j\theta+t\cdot\lambda_j=0$ with $\varepsilon_j=\pm1$. Then
$t\neq0$, and the points $\varepsilon_j\lambda_j$, for the indices with
$v_j\neq0$, are distinct points of the unit circle (no two $\lambda_j$
are antipodal) lying on the line $t\cdot x=-\theta$; so there are at
most two of them. But the third equation expresses $0$ as a positive
combination of these $\lambda_j$, which is impossible for one or two
vertices of the pentagon. If $X=0$ and $t\neq0$, all the $\lambda_j$
with $v_j\neq0$ lie on the line $t\cdot x=0$, which contains at most
one vertex, and the same contradiction applies.
\end{proof}

\subsection{The affine model}
In the chart $c=1$ write $(v_0,\ldots,v_4)$ as a $2\times5$ matrix $M$
with columns $v_j$. The induced linear vector fields are
$$
D_0:M\mapsto(H+2iI)M,\qquad E:M\mapsto EM,\qquad
D_1:M\mapsto M\Omega,\quad \Omega=\diag(\zeta^0,\ldots,\zeta^4).
$$
Thus $\widetilde G$ acts on $\C^{2\times5}$ as $B_\tau\times C_w$, where
$B_\tau=\exp(\C(H+2iI)\oplus\C E)\subset GL(2,\C)$ (the subscript
anticipates the parameter $\tau$ of Section~\ref{sec:families}; here
$\tau=i$) acts on the left
and $C_w=\{e^{w\Omega}:w\in\C\}$ on the right. Let $U\subset\C^{2\times5}$
be the image of $\cS$; then $N=U/\widetilde G$, and $U\to N$ is a
holomorphic principal $\widetilde G$-bundle by
Theorem~\ref{thm:irrational} and Remark~\ref{rem:principal-bundle}.

For a matrix $M$ of rank $2$ let $W\in\Gr(2,5)$ be its row space, with
Pl\"ucker coordinates $p_{ij}$, let
$\Poly(W)=\operatorname{Conv}\{e_i+e_j:\ p_{ij}(W)\neq0\}\subset\R^5$,
and let $\lambda:\R^5\to\R^2$, $x\mapsto\sum_jx_j\lambda_j$.

\begin{lemma}\label{lem:U}
$U=\{M:\ \operatorname{rank}M=2,\ 0\in\operatorname{int}\lambda(\Poly(W))\}$.
In particular $U$ is invariant under $GL(2,\C)$ acting on the left and
under the diagonal torus $(\C^*)^5$ acting on the right. Moreover the
set $Z_5=\{W\in\Gr(2,5):\ \sum_j\zeta^j(\Pi_W)_{jj}=0\}$, where $\Pi_W$
is the orthogonal projection onto $W$, is compact, and the map
$C_w\times Z_5\to\Gr(2,5)^s=\{W:0\in\operatorname{int}\lambda(\Poly(W))\}$,
$(w,W)\mapsto W e^{w\Omega}$, is a diffeomorphism.
\end{lemma}

\begin{proof}
By Gram--Schmidt, every $M$ of rank two is in the $GL(2,\C)$-orbit of
a matrix with $MM^*=\tfrac12I$, whose row space is the same; and for
such $M$, $\|v_j\|^2=\tfrac12(\Pi_W)_{jj}$. Since
$\widetilde S=\widetilde G\widetilde K$ and $\widetilde K$ preserves
$H^*$, $\cS=\widetilde G\cdot H^*=\widetilde S\cdot H^*$, so $U$ is the
image of $\widetilde S\cdot H^*$. In the chart, $\widetilde S$ acts by
all left multiplications by $GL(2,\C)$ and on the right by diagonal
matrices whose absolute values are those of the elements of $C_w$;
the remaining unitary diagonal factors preserve the diagonal of
$\Pi_W$. So $M\in U$ if and only if
$\operatorname{rank}M=2$ and the $C_w$-orbit of $W$ meets $Z_5$. For
$w=a+ib$ the absolute value of $e^{w\Omega}$ is $\diag(e^{\xi\cdot\lambda_j})$
with $\xi=(a,-b)$, and $w\mapsto\xi$ is a real isomorphism. The
function
$$
F_W(\xi)=\log\sum_{i<j}|p_{ij}(W)|^2e^{2\xi\cdot\lambda(e_i+e_j)}
$$
is convex, and its critical points are the $\xi$ with
$W\diag(e^{\xi\cdot\lambda_j})\in Z_5$. If $0$ lies in the interior
of $\lambda(\Poly(W))$, the support affinely spans $\R^2$; the
Hessian is a positive definite covariance form, and $F_W$ is proper.
It therefore has a unique critical point, a non-degenerate minimum.

Conversely, at a critical point the gradient equation expresses $0$
as a strictly positive convex combination of every projected weight
in the Pl\"ucker support. Thus $0$ lies in the relative interior of
the projected polytope. Suppose its affine span were proper. Since
it contains $0$, there would be $h\in\R^2\setminus\{0\}$ such that
$h\cdot(\lambda_i+\lambda_j)=0$ for every nonzero minor. At the
critical row space $W'\in Z_5$, put
$D=\diag(h\cdot\lambda_j)$. The unitary flow $e^{itD}$ fixes $W'$,
as follows from its action on the Pl\"ucker coordinates. Choose a
frame $M$ for $W'$ with $MM^*=\tfrac12I_2$. There is then a
skew-Hermitian matrix $A$ with $AM+MiD=0$. Moreover
$$
\tfrac12\operatorname{tr}A
=-i\operatorname{tr}(MDM^*)
=-i\sum_j(h\cdot\lambda_j)\|v_j\|^2=0.
$$
Hence $A\in\mathfrak{su}(2)$, and
$A+i(h_1Y_1+h_2Y_2)$ is a nonzero infinitesimal stabilizer of
$(M,1)\in H^*$, contrary to Proposition~\ref{prop:N5}.
Thus the projected polytope has full dimension and $0$ lies in its
ordinary interior. This also proves uniqueness and non-degeneracy
whenever the orbit meets $Z_5$.
The unitary part of
$e^{w\Omega}$ preserves $Z_5$. This proves the description of $U$ and the
bijectivity of the map; non-degeneracy of the minimum gives the local
diffeomorphism property, and $Z_5$ is closed in $\Gr(2,5)$. The
invariance statements hold because $\Poly(W)$ depends only on which
Pl\"ucker coordinates vanish.
\end{proof}

\begin{lemma}\label{lem:codim}
The complement of $U$ in $\C^{2\times5}$ is a closed algebraic subset
of complex codimension at least $3$.
\end{lemma}

\begin{proof}
The matrices of rank at most one form a subset of codimension $4$.
The rest of the complement is the preimage, under the submersion
$M\mapsto W$, of $\Gr(2,5)\setminus\Gr(2,5)^s$, which is closed and a
union of matroid strata since instability is preserved under
specialization. A stratum is given by the set $L$ of zero columns and
the partition of the other indices into $k\geq2$ classes of
proportional columns; it has dimension $k+(5-|L|)-4$, hence codimension
$5-k+|L|$ in $\Gr(2,5)$. For such $W$,
$\Poly(W)=\{x\in[0,1]^5:\sum x_j=2,\ x_L=0,\ \sum_{j\in A}x_j\leq1
\text{ for each class }A\}$, and the maximum of $u\cdot\lambda(x)$ on it
is $v_{(1)}(u)+v_{(2)}(u)$, the sum of the two largest class values
$v_A(u)=\max_{j\in A}u\cdot\lambda_j$. So $W$ is unstable if and only if
$v_{(1)}(u)+v_{(2)}(u)\leq0$ for some unit vector $u$.

Suppose $L=\varnothing$, and let $X$ be the class containing an index
$j_*$ maximizing $u\cdot\lambda_j$; the angle $\delta$ between $u$ and
$\lambda_{j_*}$ is at most $\pi/5$. Every $j\notin X$ satisfies
$u\cdot\lambda_j\leq-\cos\delta$, that is, $\lambda_j$ is within angle
$\delta$ of $-u$. Two vertices of the regular pentagon are within
$\pi/5$ of $-u$ only if $-u$ bisects an edge; then $u$ points to the
opposite vertex, $\delta=0$, and $u\cdot\lambda_j\leq-1$ is impossible.
So $|X|\geq4$, $k\leq2$, and the codimension is at least $3$. If
$|L|=1$, the four remaining vertices leave a gap of angle $4\pi/5$, so
$\delta\leq2\pi/5$; if all classes were singletons, the three indices
outside $X$ would lie within angle $\delta$ of $-u$, which forces $-u$
to be a vertex, and then $\delta\leq\pi/5$, a contradiction. So some
class has two elements, $k\leq3$, and the codimension is at least $3$.
If $|L|\geq2$ it is at least $4$.
\end{proof}

\subsection{Invariants}

\begin{theorem}\label{thm:N5}
The compact complex manifold $N$ of dimension $7$ has the following
properties.
\begin{enumerate}
\item $N$ is $4$-connected: $\pi_j(N)=0$ for $1\leq j\leq4$.
\item $N$ contains Hopf surfaces; in particular it is not K\"ahler.
\item Every holomorphic map from $N$ to a projective space is
constant.
\item The Lie algebra of holomorphic vector fields on $N$ is abelian
of complex dimension $4$, and the maximal compact subgroup of the
identity component $\operatorname{Aut}^0(N)$ is a torus of real
dimension $6$.
\end{enumerate}
\end{theorem}

\begin{proof}
Let $\Delta=\C^{2\times5}\setminus U$, of codimension $\geq3$ by
Lemma~\ref{lem:codim}. Then holomorphic functions on $U$ extend to
$\C^{2\times5}$ (Hartogs), $H^1(U,\mathcal O)=0$ by Scheja's theorem
\cite{Scheja}, and $H^k(U;\Z)=0$ for $1\leq k\leq4$, since $\Delta$ has real
codimension at least $6$; hence $\operatorname{Pic}(U)=0$.

(1) A Whitney stratification of $\Delta$ has real codimension at least
$6$. For $j\leq4$, any map $S^j\to U$ extends to a map of a ball
into $\C^{10}$, which can be perturbed relative to its boundary to
avoid all strata, since $j+1<6$. Thus $U$ is $4$-connected.
The principal bundle $U\to N$ has contractible fibre
$\widetilde G\cong\C^3$, so the same holds for $N$.

(2) Choose $M_0\in U$ and let $Q=\{M\in U:\ W(M)\in C_w\cdot W(M_0)\}$.
By Lemma~\ref{lem:U}, $C_w\cdot W(M_0)$ is closed in $\Gr(2,5)^s$, so $Q$
is a closed complex submanifold of $U$; it is the orbit of
$GL(2,\C)\times C_w$, which acts on it freely, again by
Lemma~\ref{lem:U}. Its image in $N$ is the compact complex surface
$Q/\widetilde G\cong B_\tau\backslash GL(2,\C)$, which is $U(2)$ with a
Samelson structure, a Hopf surface (Lemma~\ref{lem:Samelson} and the
remark following it). A complex submanifold of a K\"ahler manifold is
K\"ahler, and Hopf surfaces are not.

(3) Let $f:N\to\PP^r_{\mathbb C}$ be holomorphic, and $F=f\circ p:U\to\PP^r_{\mathbb C}$.
Since $\operatorname{Pic}(U)=0$, $F=[s_0:\cdots:s_r]$ with $s_i$
holomorphic on $U$, hence entire, without common zeros on $U$, and
the $\widetilde G$-invariance of $F$ gives
$s_i(gM)=\chi(g,M)s_i(M)$ with $\chi$ holomorphic and non-vanishing,
the same for all $i$. Differentiating along the flows of $D_0$, $D_1$
and $E$ we get $D_0s_i=\kappa_0s_i$, $D_1s_i=\kappa_1s_i$,
$Es_i=\kappa_Es_i$. For each generator $D$, define $\kappa_D$ on
$\{s_i\neq0\}\cap U$ as $(Ds_i)/s_i$. Invariance of the projective
map makes these definitions agree on overlaps. Since the $s_i$ have
no common zero on $U$, they define a holomorphic function on $U$;
Hartogs extends it to an entire function. Thus the $\kappa_\bullet$
are entire and independent of $i$. The field $D_0$ acts on a monomial in the entries of $M$, of
degrees $a$ and $b$ in the two rows, by the eigenvalue
$(1+2i)a+(-1+2i)b$, whose imaginary part is $2(a+b)$. Hence $D_0$ is
invertible, with inverse of norm at most $1/2$ on each homogeneous
component, on entire functions without constant term, and there is an
entire $\psi$ with $D_0\psi=\kappa_0-\kappa_0(0)$. Explicitly, if
$\kappa_0-\kappa_0(0)=\sum_{|\alpha|>0}c_\alpha M^\alpha$,
divide each coefficient by its monomial eigenvalue $\lambda_\alpha$.
Since $|\lambda_\alpha|\geq2|\alpha|$, the resulting power series
converges absolutely on every polydisc on which the original series
does. Replacing $s_i$ by
$s_ie^{-\psi}$, which does not change $F$, we may assume that $\kappa_0$
is constant. Then each $s_i$ is a $D_0$-eigenvector, hence a
polynomial, homogeneous of a fixed bidegree $(a,b)$ in the two rows,
since $(a,b)\mapsto(1+2i)a+(-1+2i)b$ is injective. Since $D_0$
commutes with $D_1$, $D_0\kappa_1=0$, so $\kappa_1$ is constant; and
since $[D_0,E]=-2E$ as derivations, $\kappa_E$ is a $D_0$-eigenfunction of eigenvalue
$-2$, which does not occur, so $\kappa_E=0$. The field $D_1$
multiplies a monomial whose degrees in the five columns are
$(n_0,\ldots,n_4)$ by $\sum_jn_j\zeta^j$; since $\sum n_j=a+b$ is fixed
and $5$ is prime, all the $s_i$ involve only monomials of one and the
same column multidegree $n$.

For every column index $j$, $U$ contains a matrix with that column
zero: rotate either of the stable triangular supports in
Proposition~\ref{prop:surjective-N5}. At this point some $s_i$ must
be nonzero, which forces the common column degree $n_j$ to vanish.
Thus every column degree is zero and all the $s_i$ are constant.
Consequently $F$, and hence $f$, is constant.

(4) Let $\xi$ be a holomorphic vector field on $N$. The fundamental
fields of $D_0,D_1,E$ trivialize the vertical tangent bundle of
$U\to N$, and the obstruction to lifting $\xi$ to a holomorphic field on
$U$ lies in $H^1(U,\mathcal O)^3=0$. A lift $X$ extends to an entire
vector field on $\C^{2\times5}$, and since $X$ is $p$-related to $\xi$,
$[X,D_0]=f_0D_0+f_1D_1+f_EE$ on $U$. The coefficients are
holomorphic there because the fundamental fields form a global
holomorphic frame of the vertical bundle; they extend to entire
functions by Hartogs. Decompose $X$ and
the $f_\bullet$ into homogeneous components. The operator
$\operatorname{ad}D_0$ acts on monomial vector fields of degree $d$ with
eigenvalues of imaginary part $2(d-1)$, and $\operatorname{ad}D_0(hY)
=(D_0h)Y+h[D_0,Y]$ with $[D_0,Y]\in\{0,-2E\}$ for $Y\in\{D_0,D_1,E\}$.
Hence, for $d\neq1$, the degree-$d$ component $X_d$ is obtained from the
vertical field $f_0D_0+f_1D_1+f_EE$ by applying $-(\operatorname{ad}
D_0)^{-1}$, which preserves vertical fields. More explicitly, its inverse divides
coefficients of monomial vector fields of degree $d\ne1$ by
eigenvalues of modulus at least $2|d-1|$. In the vertical frame it
acts coefficientwise through $D_0$ and $D_0-2$, with the same
bound on coefficients of degree $d-1>0$. Thus the sums of the
vertical corrections converge on every polydisc. The degree-zero
part of $X$ vanishes, since the bracket on the right has no
degree-zero term. It follows that $X-X_1$ is a vertical entire
field, and $X_1$ is a \emph{linear} lift of $\xi$.

It remains to identify which linear fields descend. A linear field vertical on $U$ has unique holomorphic coefficients
in the fundamental frame $D_0,D_1,E$. These extend by Hartogs; invariance
under scalar dilation makes them homogeneous of degree zero, hence
constant. Therefore this field belongs to $\mathfrak g$, so $[X_1,\mathfrak g]\subset\mathfrak g$. Conversely, a
linear field normalizing $\mathfrak g$ preserves $U$ (it generates left
multiplications by $GL(2,\C)$ and right multiplications by $(\C^*)^5$,
see below) and descends to $N$. Therefore the Lie algebra of
holomorphic vector fields is $\mathfrak n/\mathfrak g$, where
$\mathfrak n\subset\mathfrak{gl}(10)$ is the normalizer of $\mathfrak g$.

We compute $\mathfrak n$. Write $\C^{2\times5}=\C^2\otimes\C^5$. If
$[X,D_1]\in\mathfrak g$, then, since $\mathfrak g$ lies in the
centralizer of $D_1=I\otimes\Omega$ and $[X,D_1]$ lies in the sum of the
non-zero eigenspaces of $\operatorname{ad}D_1$, $[X,D_1]=0$; as the
$\zeta^j$ are distinct, $X=\oplus_jX_j$ acts column by column. Then
$[X,D_0]=\oplus_j[X_j,H]$ has zero diagonal, while an element
$\alpha D_0+\beta D_1+\gamma E$ has diagonal entries $\alpha(\pm1+2i)+
\beta\zeta^j$; so $\alpha=\beta=0$ and $[X_j,H]=\gamma E$ for all $j$,
which forces $X_j$ to be upper triangular with the same $(1,2)$
entry for all $j$. Finally $[X_j,E]=(p_j-s_j)E$, where $p_j,s_j$ are
the diagonal entries of $X_j$, must be the same for all $j$. Hence
$$
\mathfrak n=\{B\otimes I+I\otimes P:\ B\in\mathfrak b,\ P\text{ diagonal}\},
$$
with $\mathfrak b$ the upper triangular matrices, of dimension $3+5-1=7$.
Since $[\mathfrak n,\mathfrak n]=[\mathfrak b,\mathfrak b]\otimes I=\C E
\subset\mathfrak g$, the quotient $\mathfrak n/\mathfrak g$ is abelian,
of dimension $4$.

For completeness, the global compact-rank calculation can be made
without assuming a priori that a quotient lattice is discrete.
After dividing out the common unipotent subgroup, the effective
normalizer is the diagonal torus $T_c=(\C^*)^6$ (the two row scalars
and five column scalars, modulo their common scalar kernel).
The diagonal part of $\widetilde G$ is the exponential of a complex
plane $L$ generated by $D_0,D_1$. If its logarithmic absolute values
vanish, the difference of the two row values gives
$\operatorname{Re}s=0$, and the remaining equations are
$\operatorname{Re}(2is+w\zeta^j)=0$ for all $j$. Since the pentagon
spans an affine plane, these force $w=s=0$. Thus logarithmic absolute
values embed $L$ as a closed real four-plane. In particular $\exp L$
is closed and contractible and intersects the compact torus of $T_c$
trivially. The quotient $T_c/\exp L$ has universal cover $\C^4$ and a
discrete period group of rank $6$, so its maximal compact subgroup has
real dimension $6$. This connected group maps onto
$\operatorname{Aut}^0(N)$, which is a complex Lie group by the theorem
of Bochner and Montgomery \cite{BoMo}: its differential is the Lie-algebra
isomorphism just computed, and an open subgroup of a connected group
is the whole group. A generic matrix with at least three distinct
column directions has stabilizer in the upper Borel times the column
torus equal to the common scalar kernel. Hence the induced quotient
action has trivial generic stabilizer and no further kernel.
\end{proof}

The stable triangular supports used in (3) exist by the following fact.

\begin{proposition}\label{prop:surjective-N5}
For every $r\in\R^5_{\geq0}$ with $\sum_jr_j=1$ and $\sum_j\lambda_jr_j=0$
there is a point of $N$ with $\|v_j\|^2=r_j$ for all $j$. In
particular $N$ contains points whose only non-zero columns are
$v_0,v_2,v_3$, and points whose only non-zero columns are
$v_1,v_3,v_4$.
\end{proposition}

\begin{proof}
Taking norms in $r_j\lambda_j=-\sum_{k\neq j}r_k\lambda_k$ gives
$r_j\leq\tfrac12$. Lengths $r_j\leq\tfrac12$ with sum $1$ are the edge
lengths of a closed polygon in $\R^3$: there are unit vectors $u_j$ with
$\sum_jr_ju_j=0$. With the Pauli matrices $\sigma$, the matrices
$\tfrac{r_j}2(I+u_j\cdot\sigma)$ are positive semidefinite of rank
$\leq1$ and trace $r_j$; write them as $v_jv_j^*$. Then
$\sum_jv_jv_j^*=\tfrac12I$. The triangles $\{0,2,3\}$ and $\{1,3,4\}$
contain the centre of the pentagon in their interior, so there are
such $r$ supported on them.
\end{proof}

\begin{corollary}\label{cor:new}
The compact complex manifold $N$ of Proposition~\ref{prop:N5} is not
biholomorphic to:
\begin{enumerate}
\item any compact complex manifold with a maximal torus action in the
sense of Ishida \cite{Ishida}; in particular to no LVM or LVMB
manifold \cite{LdMV,M,Bosio,BM}, no moment-angle manifold with its
standard complex quotient structure, no classical Calabi--Eckmann
manifold, and no compact toric variety;
\item any holomorphic bundle, in the orbifold sense, with fibre a
compact Lie group with a Samelson structure over a projective orbifold, nor any principal torus bundle
over a positive-dimensional projective variety;
\item any compact complex homogeneous manifold, and in particular any
compact Lie group with a left-invariant complex structure;
\item any compact complex manifold with infinite fundamental group,
such as the Hopf, Inoue, Oeljeklaus--Toma or Vaisman manifolds.
\end{enumerate}
\end{corollary}

\begin{proof}
(1) A compact torus acting holomorphically and effectively lies in
$\operatorname{Aut}^0(N)$, so its dimension is at most $6$ by
Theorem~\ref{thm:N5}(4). A maximal action on a manifold of complex
dimension $7$ needs a torus $T$ with $\dim T+\dim T_x=14$ for some $x$,
hence $\dim T\geq7$. The listed manifolds have maximal torus actions
\cite{Ishida}.
(2) Such a manifold admits a holomorphic map onto a projective
orbifold, hence a non-constant map to a projective space if the base
has positive dimension, contradicting Theorem~\ref{thm:N5}(3); if the
base is a point, the total space is a compact Lie group with a Samelson
structure or a finite free quotient of one. Neither can be $4$-connected. Indeed,
a compact connected Lie group with nontrivial semisimple part has
nonzero $\pi_3$, which is unchanged by a finite covering; a
positive-dimensional compact connected Lie group with trivial
semisimple part is a torus, and any finite free quotient of it has
infinite fundamental group. This contradicts
Theorem~\ref{thm:N5}(1).
(3) $\operatorname{Aut}^0(N)$ has dimension $4<7$, so it has no open
orbit.
(4) Theorem~\ref{thm:N5}(1).
\end{proof}

These are exclusions of complex structures, not of underlying
diffeomorphism types. In particular, the moment-angle exclusion
concerns the standard complex structures admitting maximal torus
actions. All Hopf manifolds are excluded by their fundamental groups,
without any assertion that every Hopf manifold admits such an action.

\begin{remark}
The proof of Theorem~\ref{thm:N5}(2) shows how much of the
non-abelian part is absorbed here: the map $M\mapsto W$ induces a
holomorphic fibration of $N$ over the compact complex manifold
$Z_5=\Gr(2,5)^s/C_w$, of dimension $5$, whose fibres are the Hopf
surfaces $B_\tau\backslash GL(2,\C)$.
In the notation of Section~\ref{sec:families}, where these objects
are defined for general weights $\omega$ (here $\omega_j=\zeta^j$ and
$\tau=i$), $P_\omega=U_\omega/C_\omega\to Z_\omega$ is a holomorphic
principal $GL(2,\C)$-bundle (using inversion to switch between left
and right principal actions if desired), and
$N_{\omega,\tau}=B_\tau\backslash P_\omega$. Thus these manifolds
fit the general construction of \cite[Theorem~3.3]{PT}, with a
non-projective base. The preceding corollary excludes a projective
orbifold base, not this broader construction. The base $Z_5$ is an LVM-type
manifold built on the Grassmannian instead of a projective space: the
irrational torus acts on the multiplicity space $\C^5$ of the
isotypic component $\C^2\otimes\C^5$, its columns can vanish
independently, and its stratification is governed by the pentagon
$\{r\geq0:\sum r_j=1,\ \sum\lambda_jr_j=0\}$. The absence of
non-constant holomorphic maps from $Z_5$ to projective space follows by
composition with the surjection $N\to Z_5$. Theorem~\ref{thm:Z}
below shows that $Z_5$, and more generally the bases $Z_\omega$ of the
families of Section~\ref{sec:families}, also admit no maximal torus action, are not homogeneous, and have
no Samelson-fibre bundle presentation over a projective orbifold.
\end{remark}

\section{Families, topology and algebraic reduction}\label{sec:families}

The pentagon is a seed rather than an isolated example. We give a
criterion separating the geometric conditions needed for smoothness
from the arithmetic condition needed to control meromorphic functions.
Throughout this section $m\geq5$ and $\omega_0,\ldots,\omega_{m-1}$ are
distinct points of the unit circle, no two of which are antipodal.
We identify $\C$ with $\R^2$ and write $\lambda_j=\omega_j$.
The example of Section~\ref{sec:new} is the case $m=5$,
$\omega_j=\zeta^j$ and $\tau=i$ of the construction below; in
particular the field $D_0=H+iY_0$ used there is $H+\tau Y_0$ with
$\tau=i$.
For a rank-two matrix $M$ with row space $W$, put
$$
P_\omega(W)=\operatorname{Conv}\{\omega_i+\omega_j:p_{ij}(W)\neq0\},
\qquad
U_\omega=\{M:\operatorname{rank}M=2,\ 0\in\operatorname{int}P_\omega(W)\}.
$$
We use the following conditions:
\begin{enumerate}
\item[(G)] $\C^{2\times m}\setminus U_\omega$ has complex codimension
at least three.
\item[(R)] If $a\in\Z^m$, $\sum_ja_j=0$ and
$\sum_ja_j\omega_j=0$, then $a=0$.
\end{enumerate}
Condition (R) is an affine integral non-resonance condition; it is
stronger than mere non-closedness of the rotation group.
We shall use the following consequence of (G), denoted by (D): for
every $j$ there is a matrix in $U_\omega$ whose $j$-th column is zero.
Indeed, the linear subspace $L_j=\{v_j=0\}$ has complex codimension
two and cannot be contained in the complement of $U_\omega$, whose
codimension is at least three. Thus (D) is not an additional hypothesis.

\begin{theorem}\label{thm:families}
Suppose (G) holds, and let $\operatorname{Im}\tau>0$.
On $(\C^2)^m\oplus\C$ use the same representation as in
Section~\ref{sec:new}, with $Y_1+iY_2$ acting by $\omega_j$ on the
$j$-th column, and take
$$
\mathfrak g_{\omega,\tau}
=\C(H+\tau Y_0)\oplus\C(Y_1+iY_2)\oplus\C E.
$$
It defines a compact complex manifold $N_{\omega,\tau}$ with
properties \textup{(1)--(3)} and \textup{(5)} below. If in addition
\textup{(R)} holds, then \textup{(4)} and \textup{(6)} hold as well:
\begin{enumerate}
\item Its complex dimension is $2m-3$, and its underlying smooth
manifold is
$$
\left\{M:\ MM^*=\tfrac12I_2,\quad
\sum_j\omega_j\|v_j\|^2=0\right\}.
$$
\item It is $4$-connected and admits no symplectic form.
\item It is a holomorphic bundle of Hopf surfaces over
$Z_\omega=\Gr(2,m)^s/C_\omega$, of complex dimension $2m-5$, where
$C_\omega=\{\operatorname{diag}(e^{w\omega_j}):w\in\C\}$.
\item Every holomorphic map $N_{\omega,\tau}\to\PP^r_{\mathbb C}$ is constant.
\item Its holomorphic vector-field algebra is abelian of dimension
$m-1$, and the maximal compact subgroup of its connected automorphism
group has real dimension $m+1$.
\item Its meromorphic function field is purely transcendental of
degree $m-2$. In particular, $a(N_{\omega,\tau})=m-2$.
\end{enumerate}
Under (G) alone it has no maximal torus action and is not a compact
complex homogeneous manifold. Under (G) and (R) it is also not a
Samelson-fibre bundle over a projective orbifold, including the
possibility of finite free quotients of the model fibre.
\end{theorem}

\begin{proof}[Proof of (1)--(5) and the exclusions]
The proof of Proposition~\ref{prop:N5} applies to any distinct unit
weights without antipodal pairs. Indeed, a nonzero stabilizing
$X\in\mathfrak{su}(2)$ would force the signed weights
$\varepsilon_j\omega_j$ on the support to lie on a line. These signed
weights are distinct points of a circle, so there are at most two.
The balancing equation cannot hold on one or two non-antipodal
weights. If $X=0$, a nonzero remaining stabilizer would force the
support onto a line through the origin, again impossible.
Nonemptiness follows from (G) and the convex minimization in
Lemma~\ref{lem:U}. Thus Theorem~\ref{thm:irrational} applies.

In the chart $c=1$ the three fundamental fields are
$$
D_0M=(H+2\tau I)M,\qquad D_1M=M\operatorname{diag}(\omega_j),
\qquad D_E(M)=EM.
$$
The proof of Lemma~\ref{lem:U} is unchanged, and
$N_{\omega,\tau}=U_\omega/(B_\tau\times C_\omega)$, where
$B_\tau=\exp(\C(H+2\tau I)\oplus\C E)$.
The $C_\omega$-action on the stable Grassmannian is free and proper:
its strictly convex potential gives a global smooth slice $Z_\omega$.
The map from the frame bundle to the Grassmannian therefore descends
to a holomorphic bundle with fibre $B_\tau\backslash GL(2,\C)$,
a Hopf surface. This also verifies (3), including local triviality.

Condition (G) and relative transversality to a Whitney stratification
show that $U_\omega$ is $4$-connected. The acting group is
contractible, so $N_{\omega,\tau}$ is $4$-connected as well. In
particular $H^2(N_{\omega,\tau};\R)=0$. On a compact manifold of
positive even dimension an exact two-form cannot be symplectic,
since its top power has integral zero by Stokes' theorem. This proves
(2), strengthening the non-K\"ahler conclusion.

For (4), assume also (R). Hartogs, Scheja and the exponential sequence give
$H^1(U_\omega,\mathcal O)=0$ and $\operatorname{Pic}(U_\omega)=0$.
Lift a projective map to entire functions $s_0,\ldots,s_r$ without
common zeros on $U_\omega$. As in Theorem~\ref{thm:N5}, an entire
change of trivialization makes them $D_0$-eigenfunctions. Here the
monomial eigenvalues are
$$
(a-b)+2\tau(a+b),\qquad a,b\geq0.
$$
Their imaginary parts are $2\operatorname{Im}\tau(a+b)$, so the
inverse on positive-degree terms preserves convergence, and the
$s_i$ are polynomials of one common row bidegree. Commutation with
$D_1$ makes them eigenfunctions of $D_1$ with one common eigenvalue.
Condition (R), together with their common total degree, then gives
one common column multidegree. Condition (D) forces every component
of that multidegree to be zero. Thus all $s_i$ are constant.

For (5), the lifting argument in Theorem~\ref{thm:N5} again reduces
all holomorphic vector fields to linear normalizers: off degree one,
the relevant eigenvalues have imaginary parts
$2\operatorname{Im}\tau(d-1)$. Distinctness of the $\omega_j$ gives
$$
\mathfrak n
=\{B\otimes I+I\otimes P:B\in\mathfrak b, P\text{ diagonal}\},
\qquad \dim_\C\mathfrak n=m+2.
$$
Its derived algebra is $\C E\subset\mathfrak g_{\omega,\tau}$, so
$\mathfrak n/\mathfrak g_{\omega,\tau}$ is abelian of dimension $m-1$.
After removing the unipotent part and the common scalar kernel, its
group is $(\C^*)^{m+1}/\exp L$, where $L$ is the complex plane
spanned by the two diagonal generators. The logarithmic-absolute-value
map is injective on $L$: its kernel would give
$\operatorname{Re}s=0$ and
$\operatorname{Re}(2\tau s+w\omega_j)=0$ for every $j$; affine spanning
of the weights gives $w=s=0$. It follows that $\exp L$ is closed,
contractible and transverse to the compact torus. The quotient has
universal cover $\C^{m-1}$ and period lattice of rank $m+1$.
As in Section~\ref{sec:new}, a generic matrix with three distinct
column directions rules out any additional kernel. This calculation
uses distinctness and affine spanning of the weights, but not (R).

Finally $m+1<2m-3$ for $m\geq5$, excluding a maximal torus action;
$m-1<2m-3$ excludes homogeneity. If (R) also holds, a
positive-dimensional projective orbifold base would yield a
nonconstant holomorphic map to projective space, contrary to (4).
A zero-dimensional base would make the total space a compact Lie
group with a Samelson structure or a finite free quotient of one;
$4$-connectedness excludes both by the argument in
Corollary~\ref{cor:new}(2).
\end{proof}

\subsection{Explicit weight families}

\begin{proposition}[A pentagon and additional weights]\label{prop:seed}
For every $m\geq5$, take
$\omega_j=e^{2\pi i j/5}$ for $0\leq j\leq4$ and choose the remaining
$m-5$ weights on the unit circle, distinct and without antipodal pairs.
Then (G), and hence (D), hold. Outside a countable union of proper real-analytic
subsets of the additional angle parameters, (R) holds as well.
Consequently Theorem~\ref{thm:families} produces examples in every
odd complex dimension at least seven.
\end{proposition}

\begin{proof}
Rank-two column configurations are stratified by their $l$ zero
columns and by the partition of the others into $k$ proportionality
classes. A stratum has codimension $m-k+l$ in the full matrix space:
choose $k$ lines in $\C^2$, and a nonzero vector on its assigned line
for each of the $m-l$ nonzero columns. Restricting to the first five
columns cannot increase this codimension, as deleting a nonzero
column either removes its class or decreases the codimension by one,
and deleting a zero column decreases it by two. Moreover a stratum
of codimension at most two cannot restrict to rank at most one on
these five columns, since that determinantal condition has codimension
four. The restricted rank-two stratum therefore has codimension at
most two, and is stable by Lemma~\ref{lem:codim}. Its Pl\"ucker
support is contained in that of the full matrix, so the full matrix
is stable. Rank at most one has codimension $m-1\geq4$.
This proves (G). The unstable set is closed algebraic because it is
a finite union of closed conditions specifying vanishing minors;
stability is preserved when the support increases.

For (D), each pentagon index can be omitted by choosing a rotated
balanced triangle in the other four vertices; all additional columns
may be set to zero. For an additional index, any balanced triangle
in the pentagon suffices. The realization argument of
Proposition~\ref{prop:surjective-N5} supplies the corresponding
rank-two matrices.

The five pentagon weights already satisfy (R): an integral polynomial
of degree at most four vanishing at $e^{2\pi i/5}$ is a multiple of
$1+x+\cdots+x^4$, and its coefficients cannot sum to zero unless
that multiple is zero. Each new possible affine integral relation
involves at least one additional angle with nonzero coefficient,
and is not identically zero as that angle varies. There are countably
many relations, so their zero sets have empty interior and measure
zero. Their complement is dense and of full measure.
\end{proof}

\begin{example}[An explicit sequence in every odd dimension]
\label{ex:explicit-seed}
The generic choices in Proposition~\ref{prop:seed} can be replaced by
$$
\omega_j=e^{2\pi i j/5}\quad(0\leq j\leq4),\qquad
\omega_{4+\ell}=e^{i\ell}\quad(1\leq\ell\leq m-5).
$$
By the Hermite--Lindemann theorem, $t=e^i$ is transcendental
\cite{Baker}. An integral affine relation among these weights gives
a polynomial in $t$ with algebraic coefficients, whose constant
coefficient is a combination of fifth roots of unity. Transcendence
forces every positive-degree coefficient to vanish. The remaining
relation on the pentagon, together with the zero coefficient sum,
is trivial. Thus (R) holds. The same argument shows that these
weights are distinct and have no antipodal pairs: such an equality
would make a positive power of $t$ algebraic.
Taking $\tau=i$ yields explicit examples with invariants
$$
\begin{array}{c|c|c|c}
m&\dim_\C N&a(N)&\dim_\R T_{\max}\ \text{in }\operatorname{Aut}^0(N)\\\hline
5&7&3&6\\
6&9&4&7\\
7&11&5&8\\
8&13&6&9\\
m&2m-3&m-2&m+1
\end{array}
$$
All of them are $4$-connected, and all holomorphic maps from them to
projective spaces are constant.
\end{example}

\begin{example}[Prime polygons]\label{ex:prime-polygons}
There is also an explicit cyclotomic series without generic choices:
for every odd prime $p\geq5$ take $m=p$ and
$\omega_j=e^{2\pi i j/p}$. Then (G) and (R), and hence (D), hold.
The resulting dimensions are $7,11,19,23,31,\ldots$.
To verify (G), it suffices to consider rank-two strata of codimension
at most two. If no column is zero and the stratum is unstable,
choose a unit direction $u$ such that the two largest class maxima
of $u\cdot\omega_j$ have nonpositive sum. Let $X$ be a class containing
a vertex closest to $u$, at angular distance $\delta\leq\pi/p$.
Every vertex outside $X$ must lie within angle $\delta$ of $-u$.
This arc contains at most one vertex: two would force
$\delta=\pi/p$ and $-u$ to bisect consecutive vertices, whereas for
odd $p$ this makes $u$ itself a vertex and $\delta=0$.
Thus $|X|\geq p-1$, and the codimension is at least $p-2\geq3$,
a contradiction. If there is one zero column, codimension at most
two forces all other classes to be singletons. For $p\geq7$, any
open semicircle contains at least $(p-1)/2$ vertices before deletion,
hence at least two afterwards; the sum of the two largest projections
is strictly positive. The case $p=5$ is Lemma~\ref{lem:codim}.
Two zero columns already have codimension at least four.
Condition (D) follows by taking an enclosing triangle avoiding any
prescribed vertex: after deleting one vertex the largest angular gap
is $4\pi/p<\pi$, so the origin is interior to the remaining convex
hull; absence of antipodal pairs gives such a triangle.
Finally (R) follows from the prime cyclotomic polynomial.
The same geometric arguments apply to any regular odd polygon, but
for composite orders (R) fails and the meromorphic-function conclusion
requires a different calculation.
\end{example}

\begin{remark}[Parameters and nested examples]
The geometric property (G), and consequently (D), persists under sufficiently small
perturbations of the weights on the circle. Indeed there are finitely
many rank-two support types of codimension at most two; for each of
them the origin is strictly interior to the corresponding polygon.
The finitely many witnesses for (D) also retain their strict
stability under small perturbations; this separately verifies the
persistence of (D).
The stabilizer argument persists as long as weights remain distinct
and non-antipodal. The smooth compact level sets therefore form a
locally trivial smooth family; $\tau$ also varies freely in the upper
half-plane. Condition (R) is imposed on a dense full-measure subset,
not on an open set. We do not assert that different parameters always
give non-biholomorphic manifolds.

For the pentagon-seeded family, setting the last columns to zero
gives closed holomorphic embeddings of the smaller examples into
the larger ones, with the same $\tau$. In particular each such example
contains the original seven-dimensional manifold. Thus the family is
obtained by enlarging a representation, not by taking products.
\end{remark}

\subsection{Meromorphic functions and algebraic dimension}

We finish the proof of Theorem~\ref{thm:families} by computing the
function field. Write $v_j=(x_j,y_j)^t$ and
$p_{ij}=x_i y_j-x_j y_i$. The following functions are invariant under
all three generators and hence descend meromorphically:
$$
 t_j=\frac{p_{0j}y_1}{p_{01}y_j},\qquad 2\leq j\leq m-1.
$$
On the open set where $y_j\neq0$ and $p_{01}\neq0$, they are
$(z_0-z_j)/(z_0-z_1)$, with $z_j=x_j/y_j$.

\begin{proposition}\label{prop:function-field}
Under (G) and (R), with $\operatorname{Im}\tau>0$,
$$
\mathcal M(N_{\omega,\tau})=\C(t_2,\ldots,t_{m-1}).
$$
These $m-2$ functions are algebraically independent. In particular the
original pentagon example has algebraic dimension three. The
meromorphic map to $(\PP^1_{\mathbb C})^{m-2}$ given by the $t_j$ realizes its
algebraic reduction at the level of function fields; it cannot be
replaced by a nonconstant everywhere holomorphic map to a projective
variety.
\end{proposition}

\begin{proof}
A meromorphic function on $N_{\omega,\tau}$ pulls back to an invariant
meromorphic function $f$ on $U_\omega$. Since
$\operatorname{Pic}(U_\omega)=0$, its effective polar divisor is the
divisor of a holomorphic function $B_0$ on $U_\omega$.
Then $A_0=fB_0$ is holomorphic there. Hartogs extends both functions
to $\C^{2m}$, extending $f$ meromorphically as $A_0/B_0$.
Near the origin write $f=A/B$ with coprime convergent
power series. The identity $D_0f=0$ and factoriality of the local
analytic ring imply
$$
D_0 A=hA,\qquad D_0 B=hB
$$
for one holomorphic germ $h$. On a monomial of positive total degree
$d$, $D_0$ has eigenvalue of modulus at least
$2\operatorname{Im}\tau\,d$. Therefore
$D_0\psi=h-h(0)$ has a convergent solution with zero constant term:
in its power series, division of a degree-$d$ coefficient by the
corresponding eigenvalue multiplies its absolute value by at most
$(2\operatorname{Im}\tau\,d)^{-1}$. Absolute convergence on a
smaller polydisc is therefore preserved.
Replacing $A,B$ by $e^{-\psi}A,e^{-\psi}B$ gives two eigenfunctions
with the same eigenvalue. The eigenvalue formula forces both to be
polynomials of a common row bidegree. Hence $f$ is rational near the
origin and, by analytic continuation, on the whole affine space.
Cancel any common polynomial factors and continue to denote the
coprime numerator and denominator by $A,B$.

For any diagonal degree-preserving derivation $D$ with $Df=0$,
coprimality implies $DA=h_DA$, $DB=h_DB$ with polynomial $h_D$;
degree comparison makes $h_D$ constant. Applying this to $D_0,D_1$
shows that $A,B$ have the same row bidegree and, by (R), the same
column multidegree. Thus $f$ is invariant under all row-diagonal
and column-diagonal transformations. It is also invariant under
$E$, since it came from the quotient. Consequently it is invariant
under the upper Borel on the left and the full column torus on the
right.

The rational invariant field of this action is elementary. On the
dense open set where every $y_j$ is nonzero, column scalings normalize
$y_j=1$. The remaining effective Borel action on $z_j=x_j/y_j$ is
the common affine transformation $z_j\mapsto az_j+b$. For
$z_0\neq z_1$, normalize these two coordinates to $0$ and $1$.
The remaining normalized coordinates are exactly $t_2,\ldots,t_{m-1}$,
and are free coordinates on a rational quotient. This proves both
field generation and algebraic independence. Conversely all these
rational invariants descend to meromorphic functions because the
quotient has local holomorphic sections. Theorem~\ref{thm:families}(6) follows.

The resulting meromorphic map is dominant and induces an isomorphism
of function fields, which is the asserted algebraic reduction.
The impossibility of a nonconstant holomorphic projective map is
Theorem~\ref{thm:families}(4), so indeterminacy is essential.
\end{proof}

\begin{remark}[Holomorphic maps versus meromorphic maps]
There is no contradiction between parts (4) and (6): the functions
$t_j$ generally have indeterminacy where numerator and denominator
vanish. In particular the absence of holomorphic maps to projective
space must not be interpreted as algebraic dimension zero.
\end{remark}

\begin{remark}[Products and even dimensions]\label{rem:products}
The manifolds of Theorem~\ref{thm:families} have odd dimension. The
properties above pass formally to products, which gives manifolds
with the same properties in even dimensions. Indeed, let $N_1,\ldots,
N_r$ be manifolds as in Theorem~\ref{thm:families} whose weights satisfy (G) and (R). Their product is
$4$-connected, so it admits no symplectic form; a holomorphic map from
it to a projective space is constant on each factor slice, hence
constant; and, since compact connected complex manifolds have only
constant holomorphic functions, its Lie algebra of holomorphic vector
fields is the direct sum of those of the factors, so the maximal
compact subgroup of the identity component of its automorphism group
has dimension $\sum_i(m_i+1)<\sum_i(2m_i-3)$, and there is no maximal
torus action. Taking a factor of dimension seven and another of any
odd dimension at least seven gives every even dimension at least
fourteen, hence, together with the families themselves, every complex
dimension at least thirteen. We record this only to cover even
dimensions; the products are of course not new in any further sense.
\end{remark}

\subsection{The bases \texorpdfstring{$Z_\omega$}{Z omega}}

The bases of the Hopf-surface fibrations of
Theorem~\ref{thm:families}(3) satisfy analogous exclusions from the
specific classes considered above.

\begin{theorem}[Grassmannian quotient manifolds]\label{thm:Z}
Assume (G), and let $Z_\omega=\Gr(2,m)^s/C_\omega$, a
compact complex manifold of dimension $2m-5$. Properties
\textup{(2)} and \textup{(3)} below hold. If in addition (R) holds,
then \textup{(1)} and \textup{(4)} hold as well:
\begin{enumerate}
\item every holomorphic map from $Z_\omega$ to a projective space is
constant;
\item the Lie algebra of holomorphic vector fields on $Z_\omega$ is
abelian of dimension $m-2$, and the maximal compact subgroup of
$\operatorname{Aut}^0(Z_\omega)$ is a torus of real dimension $m-1$;
\item its first homotopy groups are
$$
\pi_1(Z_\omega)=\pi_3(Z_\omega)=0,\qquad
\pi_2(Z_\omega)\cong\pi_4(Z_\omega)\cong\Z;
$$
\item its meromorphic function field is
$$
\mathcal M(Z_\omega)=\C(q_3,\ldots,q_{m-1}),\qquad
q_j=\frac{p_{0j}p_{12}}{p_{02}p_{1j}},
$$
where the $q_j$ are algebraically independent. In particular,
$a(Z_\omega)=m-3$.
\end{enumerate}
Under (G) alone, $Z_\omega$ admits no maximal torus action and is not
homogeneous. Under (G) and (R), it is also not a Samelson-fibre bundle
over a projective orbifold, even allowing finite free quotients of
the model fibre.
\end{theorem}

\begin{proof}
Write $\Gr=\Gr(2,m)$ and $\Gr^s=\Gr(2,m)^s$. Condition (G) implies that
$\Gr\setminus\Gr^s$ has complex codimension at least three, since its
preimage under the submersion $M\mapsto W$ lies in the complement of
$U_\omega$. Hence holomorphic functions, sections of line bundles and
vector fields on $\Gr^s$ extend to $\Gr$ (Hartogs), $H^1(\Gr^s,\mathcal
O)=H^1(\Gr,\mathcal O)=0$ (Scheja \cite{Scheja}), and $H^2(\Gr^s;\Z)=
H^2(\Gr;\Z)=\Z$. By the exponential sequence every holomorphic line
bundle on $\Gr^s$ is the restriction of some $\mathcal O(k)$, and every
holomorphic function on $\Gr^s$ is constant. By the proof of
Theorem~\ref{thm:families}(3), $C_\omega\cong\C$ acts freely and
properly on $\Gr^s$ with quotient $Z_\omega$.

(1) Assume (R). A holomorphic line bundle on $Z_\omega$ pulls back to a
$C_\omega$-linearized line bundle on $\Gr^s$, that is, to some $\mathcal
O(k)$ with a lift of the action. Two lifts differ by a holomorphic
cocycle $C_\omega\times\Gr^s\to\C^*$; since holomorphic functions on
$\Gr^s$ are constant, the cocycle is a character $w\mapsto e^{\kappa w}$.
Starting from the linearization induced by the diagonal torus of
$GL(m,\C)$, the sections of the line bundle on $Z_\omega$ are therefore
the sections $s\in H^0(\Gr,\mathcal O(k))$ with
$s\circ e^{w\Omega_\omega}=e^{\kappa w}s$, where $\Omega_\omega=
\diag(\omega_0,\ldots,\omega_{m-1})$. The space $H^0(\Gr,\mathcal
O(k))$ is spanned by polynomials of degree $k$ in the Pl\"ucker
coordinates, which are weight vectors of the diagonal torus; a
monomial $\prod p_{ij}^{n_{ij}}$ has weight $a\in\Z^m_{\geq0}$ with
$\sum_ja_j=2k$, and $e^{w\Omega_\omega}$ multiplies it by
$e^{w\sum_ja_j\omega_j}$. By (R), the conditions $\sum_ja_j=2k$ and
$\sum_ja_j\omega_j=\kappa$ determine $a$; so all sections have one
common weight $a$. By (D), for each $j$ there is a point of $Z_\omega$
whose row space is contained in $\{x_j=0\}$, at which all the
Pl\"ucker coordinates $p_{ij}$ vanish. If $f:Z_\omega\to\PP^r_{\mathbb C}$ is
holomorphic and $L=f^*\mathcal O(1)$, some section of $L$ is non-zero at
that point, so $a_j=0$. Hence $a=0$, $k=0$, the sections of $L$ are
constant, and $f$ is constant.

(2) Let $\xi$ be a holomorphic vector field on $Z_\omega$. The obstruction
to lifting it to $\Gr^s$ lies in $H^1(\Gr^s,\mathcal O)=0$, and the lift
extends to a holomorphic vector field on $\Gr$, that is, an element
$X\in\mathfrak{sl}(m)$, since $\operatorname{Aut}^0(\Gr(2,m))=PGL(m,\C)$.
As $X$ is related to $\xi$, $[X,\Omega_\omega]$ is tangent to
the orbits, $[X,\Omega_\omega]=f\Omega_\omega$ on $\Gr^s$ with $f$ holomorphic,
hence constant; in $\mathfrak{gl}(m)$ this reads
$[X,\Omega_\omega]=\alpha \Omega_\omega+\beta I$. The left side has zero diagonal,
since $\Omega_\omega$ is diagonal, so $\alpha\omega_j+\beta=0$ for all $j$;
as the $\omega_j$ are distinct, $\alpha=\beta=0$. Thus $X$ commutes with
$\Omega_\omega$, whose eigenvalues are distinct, and $X$ is diagonal.
Conversely diagonal fields commute with $\Omega_\omega$ and descend. Hence
the Lie algebra of holomorphic vector fields is the quotient of the
diagonal matrices by $\C I+\C \Omega_\omega$, abelian of dimension $m-2$.

The identity component $\operatorname{Aut}^0(Z_\omega)$ is the image of
the diagonal torus $(\C^*)^m$, and the kernel of its action is
$\C^*\cdot C_\omega$: the stabilizer in $(\C^*)^m$ of a plane with all
Pl\"ucker coordinates non-zero consists of the scalars, since
$d_id_j$ must be independent of $i<j$. Write $(\C^*)^m=\C^m/2\pi i\Z^m$;
then $\operatorname{Aut}^0(Z_\omega)\cong\C^{m-2}/\Gamma$, where $\Gamma$ is
the image of $2\pi i\Z^m$ in $\C^m/(\C\mathbf 1+\C\omega)$. If
$a\in\Z^m$ lies in $\C\mathbf 1+\C\omega$, say $a=x\mathbf 1+y\omega$,
then $a_j-a_0=y(\omega_j-\omega_0)$ for all $j$. If $y\neq0$,
all the $\omega_j$ would lie on the affine real line
$\omega_0+y^{-1}\R$, which meets the unit circle in at most two
points. This contradicts distinctness of the $m\geq5$ weights.
Thus $y=0$ and $a\in\Z\mathbf 1$; (R) is not needed here. Hence $\Gamma$ has
rank $m-1$. To check discreteness directly, work first in
$T_c=(\C^*)^m/\C^*$ and consider the image of the line $\C\omega$.
If its logarithmic absolute values vanish modulo real scalars, then
$\operatorname{Re}(w\omega_j)$ is independent of $j$. The distinct
unit weights span an affine real plane, so $w=0$. The
logarithmic-absolute-value map is therefore injective on this complex
line, considered as a real plane. Its exponential is closed and
contractible in $T_c$, since bounded logarithmic absolute values
force bounded parameters. The quotient is consequently a Hausdorff
complex Lie group with universal cover $\C^{m-2}$ and discrete period
group $\Gamma$ of rank $m-1$. Its maximal compact subgroup is
$\R\Gamma/\Gamma$, of real dimension $m-1$.

(3) Choose any $\tau$ with $\operatorname{Im}\tau>0$. By
Theorem~\ref{thm:families}, there is a locally trivial holomorphic
bundle $F_\tau\to N_{\omega,\tau}\to Z_\omega$, with
$F_\tau=B_\tau\backslash GL(2,\C)$ diffeomorphic to $U(2)$.
Its homotopy groups satisfy
$$
\pi_1(F_\tau)\cong\Z,\qquad \pi_2(F_\tau)=0,\qquad
\pi_3(F_\tau)\cong\Z.
$$
Since $F_\tau$ is connected and $N_{\omega,\tau}$ is $4$-connected,
the homotopy exact sequence gives $\pi_1(Z_\omega)=0$ and
isomorphisms $\pi_j(Z_\omega)\cong\pi_{j-1}(F_\tau)$ for
$2\leq j\leq4$. This proves (3).

(4) Assume (R). Pullback by $N_{\omega,\tau}\to Z_\omega$ embeds
$\mathcal M(Z_\omega)$ in $\mathcal M(N_{\omega,\tau})$.
By Proposition~\ref{prop:function-field}, every such pullback is a
rational function of the matrix entries, invariant under the upper
Borel and the full column torus. Because it comes from $Z_\omega$,
it is also invariant under all left multiplications by $GL(2,\C)$:
these preserve the row space and therefore its image in $Z_\omega$.
Conversely, every rational function invariant under $GL(2,\C)$ and
the column torus descends meromorphically to $Z_\omega$. Indeed it
first descends along the frame bundle to $\Gr(2,m)^s$, and then along
the principal $C_\omega$-bundle to $Z_\omega$; local holomorphic
sections give the descent in both steps.

The resulting rational invariant field is the field of ordered
$m$-tuples of points of $\PP^1_{\mathbb C}$ modulo $PGL(2,\C)$. On the dense
open set where the first three column directions are distinct,
there is a unique projective transformation sending them to
$0,1,\infty$. The remaining $m-3$ coordinates freely parametrize
this rational quotient. Equivalently its field is generated by the
cross-ratios
$$
q_j=\frac{(z_0-z_j)(z_1-z_2)}{(z_0-z_2)(z_1-z_j)}
    =\frac{p_{0j}p_{12}}{p_{02}p_{1j}},\qquad 3\leq j\leq m-1,
$$
where $z_j=x_j/y_j$ when $y_j\neq0$. These are algebraically
independent, and the preceding descent argument proves the asserted
identity of fields. Thus the meromorphic map defined by the $q_j$
realizes the algebraic reduction at the level of function fields.

Finally, a compact torus acting holomorphically and effectively lies
in $\operatorname{Aut}^0(Z_\omega)$, and $m-1<2m-5$ for $m\geq5$,
so no such action is maximal; $m-2<2m-5$ excludes homogeneity.
Under (R), a positive-dimensional projective orbifold base would
give a nonconstant holomorphic map to projective space, contradicting
(1). If the base were a point, $Z_\omega$ would be a compact Lie
group with a Samelson structure or a finite free quotient of one.
Every such space has $\pi_2=0$, contradicting (3).
\end{proof}

\section*{Declaration on the use of AI}
AI tools were used to assist with the writing of this paper and
the development of Sections~\ref{sec:new} and~\ref{sec:families}.
The author takes full responsibility for the mathematical content
and the final text.

\end{document}